\documentclass{amsart}%
\usepackage{amsfonts}
\usepackage[margin=1.57in]{geometry}
\usepackage[all,cmtip]{xy}
\usepackage{amsmath}
\usepackage{hyperref}
\usepackage{amssymb}
\usepackage{graphicx}%
\theoremstyle{plain}
\newtheorem{theorem}{Theorem}
\numberwithin{theorem}{section}
\newtheorem*{acknowledgement}{Acknowledgement}

\newtheorem{corollary}[theorem]{Corollary}

\newtheorem{thmannounce}{Theorem}
\newtheorem{definition}[theorem]{Definition}

\newtheorem{lemma}[theorem]{Lemma}

\newtheorem{proposition}[theorem]{Proposition}
\theoremstyle{remark}
\newtheorem{remark}[theorem]{Remark}

\newtheorem{example}[theorem]{Example}
\numberwithin{equation}{section}
\begin{document}
\title[Relative $K$-group in the ETNC, III]{On the relative $K$-group in the ETNC, Part III}
\author{Oliver Braunling}
\address{Department of Mathematics, University of Freiburg, D-79104 Freiburg
im\ Breisgau, Germany}
\thanks{The author was supported by DFG GK1821 \textquotedblleft Cohomological Methods
in Geometry\textquotedblright.}

\begin{abstract}
The previous papers in this series were restricted to regular orders. In
particular, we could not handle integral group rings, one of the most
interesting cases of the ETNC. We resolve this issue. We obtain versions of
our main results valid for arbitrary non-commutative Gorenstein orders. This
encompasses the case of group rings. The only change we make is using a
smaller subcategory inside all locally compact modules.

\end{abstract}
\maketitle

\section{Introduction}

This paper is concerned with the non-commutative equivariant Tamagawa number
conjecture (ETNC) in the formulation of Burns and Flach \cite{MR1884523}. We
assume some familiarity with this framework and use the same notation. Let $A$
be a finite-dimensional semisimple $\mathbb{Q}$-algebra and $\mathfrak{A}%
\subset A$ an order. Using the Burns--Flach theory, a Tamagawa number is an
element%
\[
T\Omega\in K_{0}(\mathfrak{A},\mathbb{R})
\]
in the relative $K$-group $K_{0}(\mathfrak{A},\mathbb{R})$. In our previous
paper \cite{etnclca} we have proposed the following viewpoint: Originally
Tamagawa numbers were defined as volumes in terms of the Haar measure. Then we
argued that the universal determinant functor of the category of locally
compact abelian (LCA) groups is the Haar measure in a suitable sense. Thus,
when wanting to define an \textit{equivariant} Tamagawa number, one should
work with an equivariant\ Haar measure. This led us to consider the category
of $\mathfrak{A}$\textit{-equivariant} LCA groups, denoted by $\mathsf{LCA}%
_{\mathfrak{A}}$. The universal determinant functor of this category should be
a reasonable approach to an `equivariant Haar measure', and thus to
equivariant Tamagawa numbers.

Unfortunately, the above picture turned out to be true only for
\textit{regular} orders. However, in this case it works perfectly: We proved%
\[
K_{0}(\mathfrak{A},\mathbb{R})\cong K_{1}(\mathsf{LCA}_{\mathfrak{A}})\text{,}%
\]
showing that our Haar measure based philosophy leads to exactly the same group
as in the original Burns--Flach formulation. One of the most attractive cases
of the ETNC is for integral group rings $\mathfrak{A}=\mathbb{Z}[G]$, where
$G$ is a finite group. These orders are regular only for the trivial group, so
\cite{etnclca} fails to deliver in this interesting case.

In the present paper, we introduce a full subcategory%
\[
\mathsf{LCA}_{\mathfrak{A}}^{\ast}\subseteq\mathsf{LCA}_{\mathfrak{A}}%
\]
which fulfills the above picture for arbitrary Gorenstein orders
$\mathfrak{A}$. This encompasses hereditary orders (which we could also handle
previously), but more importantly group rings. Besides switching to this
smaller category, the formulation of the results remains the same:

\begin{thmannounce}
Suppose $A$ is a finite-dimensional semisimple $\mathbb{Q}$-algebra and let
$\mathfrak{A}\subset A$ be a Gorenstein order. There is a canonical long exact
sequence of algebraic $K$-groups%
\[
\cdots\rightarrow K_{n}(\mathfrak{A})\rightarrow K_{n}(A_{\mathbb{R}%
})\rightarrow K_{n}(\mathsf{LCA}_{\mathfrak{A}}^{\ast})\rightarrow
K_{n-1}(\mathfrak{A})\rightarrow\cdots
\]
for positive $n$, ending in%
\[
\cdots\rightarrow K_{0}(\mathfrak{A})\rightarrow K_{0}(A_{\mathbb{R}%
})\rightarrow K_{0}(\mathsf{LCA}_{\mathfrak{A}}^{\ast})\rightarrow
\mathbb{K}_{-1}(\mathfrak{A})\rightarrow0\text{.}%
\]
Here $\mathbb{K}_{-1}$ denotes non-connective $K$-theory. There is a canonical
isomorphism%
\[
K_{1}(\mathsf{LCA}_{\mathfrak{A}}^{\ast})\cong K_{0}(\mathfrak{A}%
,\mathbb{R})\text{,}%
\]
where $K_{0}(\mathfrak{A},\mathbb{R})$ is the relative $K$-group appearing in
the Burns--Flach formulation of the non-commutative ETNC\ in \cite{MR1884523}.
\end{thmannounce}

This will be Theorem \ref{thm_OrdinaryKGroupsLongExactSequence}.\ If
$\mathfrak{A}$ is additionally a regular order (e.g., hereditary), this
sequence agrees with the one of \cite[Theorem 11.2]{etnclca}, and moreover
$\mathbb{K}_{n}(\mathsf{LCA}_{\mathfrak{A}}^{\ast})=0$ for $n\leq-1$ in this
case. Although they have the same $K$-theory, the category $\mathsf{LCA}%
_{\mathfrak{A}}^{\ast}$ will be strictly smaller than $\mathsf{LCA}%
_{\mathfrak{A}}$ also in this case. As before, in the case $\mathfrak{A}%
=\mathbb{Z}$ the universal determinant functor is the ordinary Haar measure.
This remains true also for our smaller category $\mathsf{LCA}_{\mathbb{Z}%
}^{\ast}\subset\mathsf{LCA}_{\mathbb{Z}}$.

\begin{thmannounce}
The Haar functor $Ha:\mathsf{LCA}_{\mathbb{Z}}^{\ast\times}\rightarrow
\mathsf{Tors}(\mathbb{R}_{>0}^{\times})$ is the universal determinant functor
of the category $\mathsf{LCA}_{\mathbb{Z}}^{\ast}$. Here

\begin{enumerate}
\item for any LCA\ group $G$, $Ha(G)$ denotes the $\mathbb{R}_{>0}^{\times}%
$-torsor of all Haar measures on $G$, and

\item Deligne's Picard groupoid of virtual objects for $\mathsf{LCA}%
_{\mathbb{Z}}^{\ast}$ turns out to be isomorphic to the Picard groupoid of
$\mathbb{R}_{>0}^{\times}$-torsors.
\end{enumerate}
\end{thmannounce}

This is exactly as \cite[Theorem 12.8]{etnclca}, which was for the bigger
category $\mathsf{LCA}_{\mathbb{Z}}$. In Part II of this series
\cite{etnclca2}, we had introduced double exact sequences $\left\langle
\left\langle P,\varphi,Q\right\rangle \right\rangle $.

\begin{thmannounce}
Let $A$ be a finite-dimensional semisimple $\mathbb{Q}$-algebra and
$\mathfrak{A}\subset A$ an order. Then the map%
\begin{equation}
K_{0}(\mathfrak{A},\mathbb{R})\longrightarrow K_{1}(\mathsf{LCA}%
_{\mathfrak{A}}^{\ast}) \label{ler2a}%
\end{equation}
sending $[P,\varphi,Q]$ to the double exact sequence $\left\langle
\left\langle P,\varphi,Q\right\rangle \right\rangle $ is a well-defined
morphism from the Bass--Swan to the Nenashev presentation. If $\mathfrak{A}$
is a Gorenstein order, then this map is an isomorphism.
\end{thmannounce}

See Theorem \ref{thm_ComparisonMap}. Again, the same statement holds for the
bigger category $\mathsf{LCA}_{\mathfrak{A}}$ if $\mathfrak{A}$ is regular, as
we had shown in \cite{etnclca2}.

All this fits into a bigger picture, which we will not recall in this text.
Instead, in the manuscript \cite{etnclcaAlt} we explain an alternative
construction of the non-commutative Tamagawa numbers based on our viewpoint.
It defines the same Tamagawa numbers as Burns--Flach \cite{MR1884523}, i.e.
leads to a fully equivalent formulation, but the way the Tamagawa number is
defined is quite different.

The category $\mathsf{LCA}_{\mathfrak{A}}^{\ast}$ as well as the bigger
$\mathsf{LCA}_{\mathfrak{A}}$ are closely connected to firstly Clausen's work
on a $K$-theoretic enrichment of the Artin map \cite{clausen}, as well as the
Clausen--Scholze theory of condensed mathematics \cite{condensedmath} as well
as the pyknotic mathematics of Barwick--Haine \cite{pyknotic}.

\begin{acknowledgement}
I heartily thank B. Chow, D. Clausen, B. Drew, B. K\"{o}ck for discussions and
in part helping me with proofs and fixing problems. I thank R. Henrard and
A.-C. van Roosmalen for interesting discussions around how their technology in
\cite{hr}, \cite{hr2} might lead to a quicker proof.
\end{acknowledgement}

\section{Conventions}

In this text the word \emph{ring} refers to a unital associative (not
necessarily commutative) ring. Ring homomorphisms preserve the unit of the
ring. Unless said otherwise, modules are right modules. Given an exact
category $\mathsf{C}$, we write $\mathsf{C}^{ic}$ for the idempotent
completion, \textquotedblleft$\hookrightarrow$\textquotedblright\ for
admissible monics, \textquotedblleft$\twoheadrightarrow$\textquotedblright%
\ for admissible epics, and we generally follow the conventions of B\"{u}hler
\cite{MR2606234}.

Differing from any convention, we call objects $X\in\mathsf{C}$ in a
cocomplete category $\mathsf{C}$ \emph{categorically compact} if
$\operatorname*{Hom}_{\mathsf{C}}(X,-)$ commutes with filtered colimits.
Usually, such objects are merely called compact, but since this potentially
conflicts with the topological meaning of compact, which plays a far bigger
r\^{o}le in this text, it seems best to be careful. These objects are also
called `finitely presented', but again this could potentially cause confusion,
so it is best only to refer to the ring-theoretic concept by these terms.

\section{PI-presentations}

\begin{definition}
\label{def_CPI}Suppose $\mathsf{C}$ is an exact category. Let

\begin{enumerate}
\item $\mathsf{P}$ be a full subcategory of projective objects in $\mathsf{C}$
which is closed under finite direct sums,

\item $\mathsf{I}$ be a full subcategory of injective objects in $\mathsf{C}$
which is closed under finite direct sums.
\end{enumerate}

We write $\mathsf{C}\left\langle \mathsf{P},\mathsf{I}\right\rangle $ for the
full subcategory of objects $X\in\mathsf{C}$ such that an exact sequence%
\[
P\hookrightarrow X\twoheadrightarrow I
\]
with $P\in\mathsf{P}$ and $I\in\mathsf{I}$ exists in $\mathsf{C}$. We call any
such exact sequence a \emph{PI-presentation} for $X$.
\end{definition}

\begin{lemma}
\label{lemma_ExactSeqTo3By3Diagram}Suppose we are in the situation of
Definition \ref{def_CPI}. Assume $\mathsf{C}$ is weakly idempotent
complete\footnote{\cite[\S 7]{MR2606234}, e.g., idempotent complete}. Suppose%
\begin{equation}
X^{\prime}\hookrightarrow X\twoheadrightarrow X^{\prime\prime} \label{lmipy1}%
\end{equation}
is an exact sequence in $\mathsf{C}$ such that $X^{\prime},X^{\prime\prime}%
\in\mathsf{C}\left\langle \mathsf{P},\mathsf{I}\right\rangle $. Pick any
PI-presentations for $X^{\prime}$ and $X^{\prime\prime}$ (where we denote the
objects accordingly with a single prime or double prime superscript). Then one
can extend Sequence \ref{lmipy1} to a commutative diagram%
\begin{equation}%
\xymatrix{
P^{\prime} \ar@{^{(}->}[d] \ar@{^{(}->}[r] & P^{\prime} \oplus P^{\prime
\prime} \ar@{^{(}->}[d] \ar@{->>}[r] &
P^{\prime\prime} \ar@{^{(}->}[d] \\
X^{\prime} \ar@{->>}[d] \ar@{^{(}->}[r] & X \ar@{->>}[d] \ar@{->>}[r] &
X^{\prime\prime} \ar@{->>}[d] \\
I^{\prime} \ar@{^{(}->}[r] & I^{\prime} \oplus I^{\prime\prime} \ar@{->>}[r] &
I^{\prime\prime} \\
}
\label{lmipy3}%
\end{equation}
with exact rows and exact columns. In particular, the middle column is a
PI-presentation for $X$.
\end{lemma}

\begin{proof}
First, use the PI-presentation of $X^{\prime}$.\ We get a commutative diagram%
\[%
\xymatrix{
P^{\prime} \ar@{^{(}->}[d] \ar@{^{(}->}[dr] \\
X^{\prime} \ar@{^{(}->}[r] \ar@{->>}[d] & X \\
I^{\prime}
}%
\]
and thus the admissible filtration $P^{\prime}\hookrightarrow X^{\prime
}\hookrightarrow X$ with $P^{\prime}\in\mathsf{P}$. Noether's Lemma
(\cite[Lemma 3.5]{MR2606234}) yields the exact sequence $X^{\prime}/P^{\prime
}\hookrightarrow X/P^{\prime}\twoheadrightarrow X/X^{\prime}$, which after
unravelling the outer terms, is isomorphic to%
\[
I^{\prime}\hookrightarrow X/P^{\prime}\twoheadrightarrow X^{\prime\prime
}\text{.}%
\]
Since $I^{\prime}\in\mathsf{I}$ is injective, the sequence splits. We get%
\begin{equation}
X/P^{\prime}\cong I^{\prime}\oplus X^{\prime\prime}\text{.} \label{lmipy1a}%
\end{equation}
Next, use the PI-presentation of $X^{\prime\prime}$. The direct sum of the
exact sequences%
\begin{equation}
P^{\prime\prime}\hookrightarrow X^{\prime\prime}\overset{q^{\prime\prime}%
}{\twoheadrightarrow}I^{\prime\prime}\qquad\text{and}\qquad0\hookrightarrow
I^{\prime}\overset{1}{\twoheadrightarrow}I^{\prime} \label{lmipy1b}%
\end{equation}
is again exact. As a composition of admissible epics is an admissible epic,
the kernel $Y$ in the following commutative diagram exists.%
\begin{equation}%
\xymatrix{
Y \ar@{^{(}->}[dr] & & P^{\prime\prime} \ar@{^{(}->}[d] \\
P^{\prime} \ar@{..>}[u] \ar@{^{(}->}[r] & X \ar@{->>}[r] \ar@{->>}%
[dr] & X/{P^{\prime}}
\ar@{->>}[d]^{1 \oplus q^{\prime\prime}} \\
& & I^{\prime} \oplus I^{\prime\prime}
}
\label{lmipy2}%
\end{equation}
The right column comes from the sum of sequences in Equation \ref{lmipy1b} and
the isomorphism of Equation \ref{lmipy1a} in the middle term of the right
column. By the universal property of kernels, we obtain a unique arrow
$P^{\prime}\rightarrow Y$. Since $\mathsf{C}$ is weakly idempotent complete,
we may apply the dual of \cite[Corollary 7.7]{MR2606234} and deduce that this
arrow must be an admissible monic. Thus, we obtain the admissible filtration
$P^{\prime}\hookrightarrow Y\hookrightarrow X$ and again by Noether's Lemma
the exact sequence $Y/P^{\prime}\hookrightarrow X/P^{\prime}\overset
{a}{\twoheadrightarrow}X/Y$. Unravelling the right term, this exact sequence
is isomorphic to%
\[
Y/P^{\prime}\hookrightarrow X/P^{\prime}\twoheadrightarrow I^{\prime}\oplus
I^{\prime\prime}\text{.}%
\]
Inspecting Diagram \ref{lmipy2} note that under the isomorphism of Equation
\ref{lmipy1a} the map $a$ is identified with $1\oplus q^{\prime\prime}$. Thus,
$Y/P^{\prime}$ is a kernel of this, and thus isomorphic to $P^{\prime\prime}$.
Hence, $P^{\prime}\hookrightarrow Y\twoheadrightarrow Y/P^{\prime}$ is
isomorphic to $P^{\prime}\hookrightarrow Y\twoheadrightarrow P^{\prime\prime}%
$, which splits since $P^{\prime\prime}\in\mathsf{P}$ is projective, and thus
$Y\cong P^{\prime}\oplus P^{\prime\prime}$. Then the diagonal exact sequence
of Diagram \ref{lmipy2} is a PI-presentation, and moreover the one in our
claim. Going through the maps which we have constructed, we obtain all the
arrows in Diagram \ref{lmipy3}.
\end{proof}

\begin{corollary}
\label{cor_CPICat}Suppose we are in the situation of Definition \ref{def_CPI}
and $\mathsf{C}$ is weakly idempotent complete. Then $\mathsf{C}\left\langle
\mathsf{P},\mathsf{I}\right\rangle $ is extension-closed in $\mathsf{C}$. In
particular, it is a fully exact subcategory of $\mathsf{C}$.
\end{corollary}

\begin{proof}
The lemma shows that $X$ also has a PI-presentation, so $X\in\mathsf{C}%
\left\langle \mathsf{P},\mathsf{I}\right\rangle $.
\end{proof}

\begin{lemma}
\label{lemma_ProjectivesStayProjectiveInCPI}If $X\in\mathsf{C}\left\langle
\mathsf{P},\mathsf{I}\right\rangle $ is injective (resp. projective) as an
object in $\mathsf{C}$, it is also injective (resp. projective) as an object
in $\mathsf{C}\left\langle \mathsf{P},\mathsf{I}\right\rangle $.
\end{lemma}

\begin{proof}
Immediate.
\end{proof}

In particular, all objects of $\mathsf{P}$ are still projective in
$\mathsf{C}\left\langle \mathsf{P},\mathsf{I}\right\rangle $ and
correspondingly for the injectives in $\mathsf{I}$.

\section{Construction of the category $\mathsf{PLCA}_{\mathfrak{A}}$}

Suppose $A$ is a finite-dimensional semisimple $\mathbb{Q}$-algebra and
$\mathfrak{A}\subset A$ an order. We shall use the category $\mathsf{LCA}%
_{\mathfrak{A}}$ of \cite{etnclca}. We recall that its

\begin{enumerate}
\item objects are locally compact topological right $\mathfrak{A}$-modules, and

\item morphisms are continuous $\mathfrak{A}$-module homomorphisms.
\end{enumerate}

An admissible monic is a closed injective morphism, an admissible epic is an
open surjective morphism. This makes $\mathsf{LCA}_{\mathfrak{A}}$ a
quasi-abelian exact category, generalizing an observation due to
Hoffmann--Spitzweck \cite{MR2329311}.

\begin{proposition}
\label{prop_LCACat}The category $\mathsf{LCA}_{\mathfrak{A}}$ is a
quasi-abelian exact category. There is an exact functor%
\[
(-)^{\vee}:\mathsf{LCA}_{\mathfrak{A}}^{op}\longrightarrow\left.
\mathsf{LCA}_{\mathfrak{A}^{op}}\right.  \qquad\qquad M\longmapsto
\operatorname*{Hom}(M,\mathbb{T})\text{,}%
\]
where the continuous right $\mathfrak{A}$-module homomorphism group
$\operatorname*{Hom}(M,\mathbb{T})$ is equipped with the compact-open topology
(that is: on the level of the underlying LCA group $(M;+)$ this is the
Pontryagin dual), and the left action%
\begin{equation}
(\alpha\cdot\varphi)(m):=\varphi(m\cdot\alpha)\qquad\text{for all}\qquad
\alpha\in\mathfrak{A}\text{, }m\in M\text{.} \label{lmf1}%
\end{equation}
There is a natural equivalence of functors from the identity functor to double
dualization,%
\[
\eta:\operatorname*{id}\longrightarrow(-)^{\vee}\circ\left[  (-)^{\vee
}\right]  ^{op}\text{.}%
\]
In other words: For every object $M\in\mathsf{LCA}_{\mathfrak{A}}$ there
exists a reflexivity isomorphism $\eta(M):M\overset{\sim}{\longrightarrow
}M^{\vee\vee}$, and the isomorphisms $\eta(M)$ are natural in $M$.
\end{proposition}

See \cite[Proposition 3.5]{etnclca}. If $\mathfrak{A}$ is commutative, it is
even an exact category with duality in the sense of \cite[Definition
2.1]{MR2600285}.

Let $R$ be a ring. We write $P(R)$ for the category of all projective right
$R$-modules, and $P_{f}(R)$ for the finitely generated projective right
$R$-modules. These are both exact categories in the standard way. These
categories are idempotent complete and split exact.

Write $P_{\oplus}(R)$ for the full subcategory of $P(R)$ whose objects are at
most countable direct sums of objects in $P_{f}(R)$. This is an
extension-closed full subcategory and thus itself an exact category. This
category may also be realized as%
\begin{equation}
P_{\oplus}(R)=\mathsf{Ind}_{\aleph_{0}}^{a}(P_{f}(R))\text{,} \label{laf1}%
\end{equation}
because by \cite[Corollary 3.19]{MR3510209} it is the full subcategory of
countable direct sums of objects in $P_{f}(R)$ inside $\mathsf{Lex}(P_{f}(R))$
and by \cite[Lemma 2.21]{MR3510209} the latter category is $\mathsf{Mod}(R)$.

The following is (in different formulation) due to Akasaki and Linnell.

\begin{lemma}
[Akasaki--Linnell]\label{lemma_SolvableGroupInfiniteIndecompProjective}Suppose
$G$ is a finite group and $R:=\mathbb{Z}[G]$. Then $P_{\oplus}(R)$ is
idempotent complete if and only if $G$ is solvable.
\end{lemma}

\begin{proof}
By Equation \ref{laf1} and \cite[Proposition 3.25]{MR3510209} the idempotent
completion of $P_{\oplus}(R)$ is the category $P_{\aleph_{0}}(R)$ of at most
countably generated projective $R$-modules. If $G$ is solvable, Swan
\cite[Theorem 7]{MR0153722} has shown that every projective $R$-module is
either finitely generated or free (or both), so each such is a direct sum of
finitely generated projectives, hence lies in $P_{\oplus}(R)$. On the other
hand, if $G$ is non-solvable, Akasaki exhibits a non-zero countably generated
projective $R$-module $P\in P_{\aleph_{0}}(R)$ with trace ideal $\tau
(M)\subsetneqq\mathbb{Z}[G]$, see \cite[Theorem]{MR671200} (or Linnell
\cite{MR647193}). If $P$ has a non-zero finitely generated projective summand
$P^{\prime}\subset P$, then $\tau(P^{\prime})=\mathbb{Z}[G]$ by
\cite[Corollary 1.4]{MR0304420}, and thus we would have $\tau(P)=\mathbb{Z}%
[G]$ because all maps from a direct summand extend to maps of all of $P$.
However, the latter is impossible by Akasaki's construction. Thus, $P$ has no
finitely generated projective summands and thus $P\notin P_{\oplus}(R)$.
\end{proof}

Note that $P_{\oplus}(\mathfrak{A})$ lies inside $\mathsf{LCA}_{\mathfrak{A}}$
when being regarded as a full subcategory of objects with the discrete
topology. Define $I_{\Pi}(\mathfrak{A})$ as the Pontryagin dual of $P_{\oplus
}(\mathfrak{A}^{op})$. In other words, this is the category of at most
countable products $\prod P_{i}^{\vee}$, where $P_{i}\in P_{f}(\mathfrak{A}%
^{op})$. Under Pontryagin duality these projective left $\mathfrak{A}$-modules
(i.e. right $\mathfrak{A}^{op}$-modules) become injective right $\mathfrak{A}%
$-modules in $\mathsf{LCA}_{\mathfrak{A}}$.

Define%
\begin{equation}
\mathsf{PLCA}_{\mathfrak{A}}:=\mathsf{LCA}_{\mathfrak{A}}\left\langle
P_{\oplus}(\mathfrak{A}),I_{\Pi}(\mathfrak{A})\right\rangle \text{.}
\label{laa0}%
\end{equation}
Since $\mathsf{LCA}_{\mathfrak{A}}$ is quasi-abelian, it is in particular
weakly idempotent complete and thus $\mathsf{PLCA}_{\mathfrak{A}}$ is a fully
exact subcategory of $\mathsf{LCA}_{\mathfrak{A}}$ by Corollary
\ref{cor_CPICat}.

We get a natural extension of Proposition \ref{prop_LCACat}.

\begin{proposition}
\label{prop_PontryaginDualityOnPLCA}The category $\mathsf{PLCA}_{\mathfrak{A}%
}$ is an exact category. The exact Pontryagin duality functor $(-)^{\vee}$ of
Proposition \ref{prop_LCACat} restricts to an exact equivalence of exact
categories%
\begin{align*}
(-)^{\vee}:\mathsf{PLCA}_{\mathfrak{A}}^{op}  &  \longrightarrow
\mathsf{PLCA}_{\mathfrak{A}^{op}}\\
M  &  \longmapsto\operatorname*{Hom}(M,\mathbb{T})\text{.}%
\end{align*}
We usually regard the objects of $\mathsf{PLCA}_{\mathfrak{A}^{op}}$ as
topological left $\mathfrak{A}$-modules. If $\mathfrak{A}$ is commutative,
$\mathfrak{A}=\mathfrak{A}^{op}$, and this functor makes $\mathsf{PLCA}%
_{\mathfrak{A}}$ an exact category with duality.
\end{proposition}

\begin{proof}
If $P\hookrightarrow X\twoheadrightarrow I$ is a\ PI-presentation for $X$, the
duality functor sends it to%
\[
I^{\vee}\hookrightarrow X^{\vee}\twoheadrightarrow P^{\vee}\text{,}%
\]
but by construction $I^{\vee}\in P_{\oplus}(\mathfrak{A}^{op})$ and $P^{\vee
}\in I_{\Pi}(\mathfrak{A}^{op})$, giving a PI-presentation of $X^{\vee}$.
\end{proof}

\begin{lemma}
\label{lemma_ObjectsInIProdCompactConnected}All objects in $I_{\Pi
}(\mathfrak{A})$ are compact\footnote{in the sense of topology} connected.
\end{lemma}

\begin{proof}
We use that $I_{\Pi}(\mathfrak{A})$ is the Pontryagin dual to $P_{\oplus
}(\mathfrak{A}^{op})$. Each object $P\in P_{\oplus}(\mathfrak{A}^{op})$ is
discrete, so $P^{\vee}\in I_{\Pi}(\mathfrak{A})$ is compact. As $P$ is
projective, it is also $\mathbb{Z}$-torsionfree, and thus $P^{\vee}$ is
connected by \cite[Corollary 1 to Theorem 31]{MR0442141}.
\end{proof}

The following observation is trivial.

\begin{lemma}
\label{lemma_FinGenSubmodulesLieInFigGenProjectiveSubmodule}Suppose $P\in
P_{\oplus}(\mathfrak{A})$. If $F$ is a finitely generated submodule of $P$,
then there exists a direct sum splitting%
\begin{equation}
P\cong P_{0}\oplus P_{\infty} \label{laa2a}%
\end{equation}
with $P_{0}\in P_{f}(\mathfrak{A})$, $P_{\infty}\in P_{\oplus}(\mathfrak{A})$
and $F\subseteq P_{0}$. In other words: Every finitely generated
$\mathfrak{A}$-submodule of $P$ is contained in a finitely generated
projective direct summand of $P$.
\end{lemma}

\begin{proof}
Write $P=\bigoplus_{i\in\mathcal{I}}P_{i}$ with $P_{i}\in P_{f}(\mathfrak{A}%
)$. Let $m_{1},\ldots,m_{n}$ be $\mathfrak{A}$-module generators of $F$. Since
$F\subseteq P$, we can write $m_{j}=\sum\alpha_{j,i}$ such that $\alpha
_{j,i}\in P_{i}$ and these are finite sums. Hence, collecting all the indices
$i$ which occur in these finite sums where $j=1,\ldots,n$, we get a finite
subset $\mathcal{I}_{0}$ of indices within $\mathcal{I}$. Define%
\[
P_{0}:=\bigoplus_{i\in\mathcal{I}_{0}}P_{i}\qquad\text{and}\qquad P_{\infty
}:=\bigoplus_{i\in\mathcal{I}\setminus\mathcal{I}_{0}}P_{i}\text{.}%
\]
Then $P\simeq P_{0}\oplus P_{\infty}$ as desired, $P_{0}\in P_{f}%
(\mathfrak{A})$ because $\mathcal{I}_{0}$ is finite, and $F\subseteq P_{0}$.
\end{proof}

\begin{example}
The property discussed in the previous lemma would in general be false if $P$
were allowed to be an arbitrary (countably generated) projective module. For
example, if $G$ is a non-solvable finite group, by Lemma
\ref{lemma_SolvableGroupInfiniteIndecompProjective} one can find a countably
generated indecomposable projective. Since it admits no non-trivial direct sum
decompositions at all, no splitting as in Equation \ref{laa2a} can exist.
\end{example}

\begin{lemma}
\label{lemma_SplitOffInfiniteProjectivePart}Suppose $X\in\mathsf{PLCA}%
_{\mathfrak{A}}$ has the PI-presentation%
\begin{equation}
P\hookrightarrow X\twoheadrightarrow I\text{.} \label{laa1}%
\end{equation}
Then for any finitely generated $\mathfrak{A}$-module $F\subseteq P$ there exists

\begin{enumerate}
\item a direct sum splitting%
\[
P\cong P_{0}\oplus P_{\infty}%
\]
with $F\subseteq P_{0}$, $P_{0}\in P_{f}(\mathfrak{A})$ and $P_{\infty}\in
P_{\oplus}(\mathfrak{A})$, and

\item a direct sum splitting%
\[
X\cong M\oplus P_{\infty}%
\]
with $M\in\mathsf{PLCA}_{\mathfrak{A}}$ such that $P_{0}\hookrightarrow
M\twoheadrightarrow I$ is a PI-presentation for $M$.
\end{enumerate}
\end{lemma}

It might be worth unpacking what we are saying here: Given any object $X$ and
any finitely generated submodule in $P$, we can up to a direct summand from
$P_{\oplus}(\mathfrak{A})$ isomorphically replace $X$ by an object whose
PI-presentation has only a finitely generated $P$, and we can demand that the
given $F$ lies entirely in this $P$.

\begin{proof}
By \cite[Lemma 6.5]{etnclca} in the bigger category $\mathsf{LCA}%
_{\mathfrak{A}}$ we get an exact sequence%
\begin{equation}
V\oplus C\hookrightarrow X\twoheadrightarrow D \label{laa3a}%
\end{equation}
with $V$ a vector $\mathfrak{A}$-module, $C$ a compact $\mathfrak{A}$-module
and $D$ a discrete $\mathfrak{A}$-module. Define%
\begin{equation}
J:=P\cap(V\oplus C) \label{laa3}%
\end{equation}
in $\mathsf{LCA}_{\mathfrak{A}}$. Note that both $P$ and $V\oplus C$ are
closed in $X$. As $J$ is closed in $P$, $J$ is discrete. Further, since $P$ is
a projective $\mathfrak{A}$-module, it is $\mathbb{Z}$-torsionfree, so $J$ is
$\mathbb{Z}$-torsionfree as well. As $J$ is closed in $V\oplus C$, its
underlying LCA group must be $\mathbb{Z}^{b}$ for some $b\in\mathbb{Z}_{\geq
0}$ (reason: If $J\hookrightarrow V\oplus C$, then $V^{\vee}\oplus C^{\vee
}\twoheadrightarrow J^{\vee}$ under Pontryagin duality. Here $V^{\vee}\oplus
C^{\vee}$ is a vector module plus a discrete module. All quotients of such
must be $\mathbb{R}^{a}\oplus\mathbb{T}^{b}\oplus\tilde{D}$ with $\tilde{D}$
discrete as an LCA group by \cite[Corollary 2 to Theorem 7]{MR0442141}.
Dualizing back, the underlying LCA group of $J$ must be $\mathbb{R}^{a}%
\oplus\mathbb{Z}^{b}\oplus\tilde{C}$ with $\tilde{C}$ compact. As we already
know that $J$ is discrete and torsionfree, we must have $a=0$ and $\tilde
{C}=0$). Combining these facts, $J$ is a discrete $\mathfrak{A}$-module with
underlying LCA group $\mathbb{Z}^{b}$. It follows that $J$ is a finitely
generated $\mathfrak{A}$-submodule of $P$. Next, define%
\[
J^{\prime}:=J+F\text{.}%
\]
This is still a finitely generated $\mathfrak{A}$-submodule of $P$. Thus, by
Lemma \ref{lemma_FinGenSubmodulesLieInFigGenProjectiveSubmodule} we can find a
direct sum splitting%
\begin{equation}
P\simeq P_{0}\oplus P_{\infty} \label{laa4}%
\end{equation}
with $J^{\prime}\subseteq P_{0}$ and $P_{0}\in P_{f}(\mathfrak{A})$. In the
category $\mathsf{LCA}_{\mathfrak{A}}$ we define%
\begin{equation}
M:=(V\oplus C)+P_{0}\qquad\text{inside}\qquad X\text{.} \label{laa2}%
\end{equation}
Since $D$ in Equation \ref{laa3a} was discrete, $V\oplus C$ is an open
submodule of $X$. Thus, the sum defining $M$ is also an open submodule, thus
clopen. It follows that the inclusion $M\hookrightarrow X$ is an open
admissible monic in $\mathsf{LCA}_{\mathfrak{A}}$. Both $P_{\infty}$ and $M$
are closed submodules of $X$. We claim that%
\[
P_{\infty}\cap M=0\text{.}%
\]
(Proof: Suppose $x\in P_{\infty}\cap M$. As $x$ lies in $M$, we can write
$x=x_{vc}+x_{0}$ with $x_{vc}\in V\oplus C$ and $x_{0}\in P_{0}$ by Equation
\ref{laa2}. Hence, $x_{vc}=x-x_{0}$. As $x\in P_{\infty}\subseteq P$ and
$x_{0}\in P_{0}\subseteq P$, we find $x_{vc}\in P$. Thus, $x_{vc}\in
P\cap(V\oplus C)$ and thus $x_{vc}\in J$ by Equation \ref{laa3}. As
$J\subseteq P_{0}$ by Equation \ref{laa4}, we obtain $x_{vc}\in P_{0}$. It
follows that $x\in P_{0}$. We also have $x\in P_{\infty}$ by assumption and
therefore $x\in P_{0}\cap P_{\infty}=0$, giving the claim.) Thus, $M$ and
$P_{\infty}$ are closed submodules of $X$ with trivial intersection. We get an
exact sequence%
\[
M\oplus P_{\infty}\hookrightarrow X\twoheadrightarrow Q
\]
for some quotient $Q$ in $\mathsf{LCA}_{\mathfrak{A}}$. As $P\subseteq M\oplus
P_{\infty}$, it follows that $Q$ is an admissible quotient of $I$ by Equation
\ref{laa1}. Since $I$ is (compact) connected by Lemma
\ref{lemma_ObjectsInIProdCompactConnected}, so must be $Q$. On the other hand,
since $M$ is open (or: since it contains $V\oplus C$), $Q$ is also necessarily
discrete. Being both connected and discrete, we must have $Q=0$. We get%
\begin{equation}
X\simeq M\oplus P_{\infty} \label{laa5}%
\end{equation}
in $\mathsf{LCA}_{\mathfrak{A}}$. Next, by Noether's Lemma (\cite[Lemma
3.5]{MR2606234}) the admissible filtration%
\[
P_{\infty}\hookrightarrow P\hookrightarrow X
\]
gives rise to the exact sequence%
\[
P/P_{\infty}\hookrightarrow X/P_{\infty}\twoheadrightarrow X/P\text{.}%
\]
We have $P/P_{\infty}\cong P_{0}$ from Equation \ref{laa4}, $X/P\cong I$ from
Equation \ref{laa1}, and $X/P_{\infty}\cong M$ by Equation \ref{laa5}. Thus,
$P_{0}\hookrightarrow M\twoheadrightarrow I$ is exact. Since $P_{0}\in
P_{f}(\mathfrak{A})$ and $I\in I_{\Pi}(\mathfrak{A})$, we deduce
$M\in\mathsf{PLCA}_{\mathfrak{A}}$ from Equation \ref{laa0}. Finally, since
$P_{\infty}\in P_{\oplus}(\mathfrak{A})$, Equation \ref{laa5} is not only a
direct sum splitting in $\mathsf{LCA}_{\mathfrak{A}}$, but even in the fully
exact subcategory $\mathsf{PLCA}_{\mathfrak{A}}$. Finally, $F\subseteq P_{0}$
holds by construction.
\end{proof}

The previous result implies that the objects of $\mathsf{PLCA}_{\mathfrak{A}}$
can, up to direct summands from $P_{\oplus}(\mathfrak{A})$ and $I_{\Pi
}(\mathfrak{A})$, be reduced to such where the PI-presentation is made from
finitely generated discrete projectives and their Pontryagin duals.

\begin{proposition}
\label{prop_PLCAStructureThm}Every object in $\mathsf{PLCA}_{\mathfrak{A}}$ is
isomorphic to an object of the shape%
\[
X\simeq P_{\infty}\oplus I_{\infty}\oplus B
\]
with $P_{\infty}\in P_{\oplus}(\mathfrak{A})$, $I_{\infty}^{\vee}\in
P_{\oplus}(\mathfrak{A}^{op})$ and $B\in\mathsf{PLCA}_{\mathfrak{A}}$ has a
PI-presentation%
\[
P_{0}\hookrightarrow B\twoheadrightarrow I_{0}%
\]
with $P_{0}\in P_{f}(\mathfrak{A})$, $I_{0}^{\vee}\in P_{f}(\mathfrak{A}%
^{op})$.
\end{proposition}

\begin{proof}
Let $X\in\mathsf{PLCA}_{\mathfrak{A}}$ be any object. Pick a PI-presentation
$P\hookrightarrow X\twoheadrightarrow I$. We apply Lemma
\ref{lemma_SplitOffInfiniteProjectivePart} with $F=0$. We get a direct sum
splitting $X\simeq M\oplus P_{\infty}$ in $\mathsf{PLCA}_{\mathfrak{A}}$,
where $M$ has a PI-presentation of the shape%
\[
P_{0}\hookrightarrow M\twoheadrightarrow I
\]
such that $P_{0}\in P_{f}(\mathfrak{A})$. Now apply Pontryagin duality, giving
the exact sequence
\[
I^{\vee}\hookrightarrow M^{\vee}\twoheadrightarrow P_{0}^{\vee}%
\]
in $\mathsf{PLCA}_{\mathfrak{A}^{op}}$. This is a PI-presentation in
$\mathsf{PLCA}_{\mathfrak{A}^{op}}$. Now apply Lemma
\ref{lemma_SplitOffInfiniteProjectivePart} (again with $F=0$). Then dualize back.
\end{proof}

We recall the following standard concept from the theory of topological groups.

\begin{definition}
A subset $U$ of a topological group $G$ is called \emph{symmetric} if it is
closed under taking inverses. A topological group $G$ is called
\emph{compactly generated} if there exists a compact symmetric neighbourhood
$U\subseteq G$ of the neutral element such that $G=\bigcup_{n\geq1}U^{n}$.
\end{definition}

\begin{remark}
Unfortunately, the word \textquotedblleft compactly
generated\textquotedblright\ is also used with a different meaning elsewhere.
Either in a category-theoretic sense related to categorically compact objects,
or in a further topological meaning, probably most familiar in the setting of
compactly generated Hausdorff spaces in homotopy theory; e.g.,
\cite{condensedmath} uses both of these other meanings. This is most
unfortunate, but all uses of these words are well-established in their
respective community of mathematics.
\end{remark}

Let $\mathsf{PLCA}_{\mathfrak{A},cg}$ be the full subcategory of
$\mathsf{PLCA}_{\mathfrak{A}}$ of compactly generated $\mathfrak{A}$-modules,%
\begin{equation}
\mathsf{PLCA}_{\mathfrak{A},cg}:=\mathsf{PLCA}_{\mathfrak{A}}\cap
\mathsf{LCA}_{\mathfrak{A},cg}\text{.} \label{l_cat_cg}%
\end{equation}
Since compactly generated topological modules groups are closed under
extension in $\mathsf{LCA}_{\mathfrak{A}}$ (\cite[Corollary 7.2]{etnclca}),
this is an extension-closed subcategory of $\mathsf{PLCA}_{\mathfrak{A}}$.

\begin{lemma}
We have $\mathsf{PLCA}_{\mathfrak{A},cg}=\mathsf{LCA}_{\mathfrak{A}%
}\left\langle P_{f}(\mathfrak{A}),I_{\Pi}(\mathfrak{A})\right\rangle $, i.e.
the same category can also be described as the full subcategory of objects in
$\mathsf{PLCA}_{\mathfrak{A}}$ which admit a PI-presentation%
\[
P\hookrightarrow X\twoheadrightarrow I
\]
with $P$ finitely generated projective.
\end{lemma}

\begin{proof}
(Step 1) Suppose $X$ lies in $\mathsf{LCA}_{\mathfrak{A}}\left\langle
P_{f}(\mathfrak{A}),I_{\Pi}(\mathfrak{A})\right\rangle $. Then%
\[
P\hookrightarrow X\twoheadrightarrow I
\]
is exact with $P$ finitely generated projective and $I\in I_{\Pi}%
(\mathfrak{A})$. By Lemma \ref{lemma_ObjectsInIProdCompactConnected} the
module $I$ is compact, hence compactly generated, and $P$ has $\mathbb{Z}^{n}$
for some finite $n\geq0$ as its underlying LCA group, so it is compactly
generated, too. Thus, $X$ is an extension of compactly generated LCA\ groups,
and thus $X\in\mathsf{PLCA}_{\mathfrak{A},cg}$. (Step 2) Conversely, suppose
$X\in\mathsf{PLCA}_{\mathfrak{A},cg}$. Proposition \ref{prop_PLCAStructureThm}
gives a direct sum splitting $X=P_{\infty}\oplus M\oplus I_{\infty}$. By Step
1 we know that $M$ is compactly generated and $I_{\infty}$ is compact, so $X$
is compactly generated if and only if $P_{\infty}$ is. However, the underlying
LCA group of $P_{\infty}$ is $\bigoplus\mathbb{Z}$, over some index set, and
this is compactly generated only if $P_{\infty}$ is finitely generated.
\end{proof}

\begin{proposition}
\label{prop_LF1}The inclusion $P_{f}(\mathfrak{A})\hookrightarrow P_{\oplus
}(\mathfrak{A})$ is left $s$-filtering.\footnote{This concept originates from
the work of Schlichting \cite{MR2079996}. We use the formulation of
\cite[\S 2.2.2]{MR3510209}.}
\end{proposition}

\begin{proof}
\textit{(Left filtering)}\ Suppose we are given an arrow $g:Y\rightarrow X$
with $Y\in P_{f}(\mathfrak{A})$ and $X\in P_{\oplus}(\mathfrak{A})$. The
set-theoretic image of $Y$ in $X$ is again a finitely generated module, so
by\ Lemma \ref{lemma_FinGenSubmodulesLieInFigGenProjectiveSubmodule} we find a
direct sum decomposition
\[
X\cong P_{0}\oplus P_{\infty}%
\]
with $P_{0}\in P_{f}(\mathfrak{A})$, $P_{\infty}\in P_{\oplus}(\mathfrak{A})$
and $\operatorname*{im}_{\mathsf{Set}}(g)\subseteq P_{0}$. It follows that the
arrow $g$ factors as $Y\rightarrow P_{0}\hookrightarrow X$, showing the left
filtering property.\newline\textit{(Left special)} Suppose
$e:X\twoheadrightarrow X^{\prime\prime}$ is an admissible epic with $X\in
P_{\oplus}(\mathfrak{A})$ and $X^{\prime\prime}\in P_{f}(\mathfrak{A})$. As
$X^{\prime\prime}$ is projective, the epic splits. We obtain a diagram%
\[%
\xymatrix{
0 \ar@{^{(}->}[r] \ar[d] & 0 \oplus X^{\prime\prime} \ar@{->>}[r] \ar
[d] & X^{\prime\prime} \ar@{=}[d] \\
X^{\prime} \ar@{^{(}->}[r] & X \ar@{->>}[r] & X^{\prime\prime}
}%
\]
showing the left special property.
\end{proof}

\begin{proposition}
\label{prop_LF2}The inclusion $\mathsf{PLCA}_{\mathfrak{A},cg}\hookrightarrow
\mathsf{PLCA}_{\mathfrak{A}}$ is left $s$-filtering.
\end{proposition}

\begin{proof}
\textit{(Left filtering)} Suppose we are given an arrow $Y\rightarrow X$ with
$Y\in\mathsf{PLCA}_{\mathfrak{A},cg}$ and $X\in\mathsf{PLCA}_{\mathfrak{A}}$.
We apply Proposition \ref{prop_PLCAStructureThm} to $X$ and get the diagram%
\[%
\xymatrix{
& M \oplus I_{\infty} \ar@{^{(}->}[d] \\
Y \ar[dr]_{h} \ar[r] & P_{\infty} \oplus M \oplus I_{\infty} \ar@{->>}[d] \\
& P_{\infty}.
}%
\]
We first work entirely on the level of $\mathsf{LCA}_{\mathbb{Z}}$: Since $Y$
is compactly generated, we get some isomorphism $Y\simeq C\oplus\mathbb{Z}%
^{n}\oplus\mathbb{R}^{m}$ for some $n,m$ and $C$ compact, \cite[Theorem
2.5]{MR0215016}. As $C$ is compact, its set-theoretic image under $h$ is
compact, but since $P_{\infty}$ is discrete and torsionfree, $h(C)$ must be
zero. Moreover, the set-theoretic image of $\mathbb{R}^{m}$ under $h$ is
connected and thus also zero. It follows that the set-theoretic image of $h$
agrees with the image $h(\mathbb{Z}^{n})$, and thus must be a finitely
generated $\mathbb{Z}$-submodule of $P_{\infty}$. Now return to $\mathsf{LCA}%
_{\mathfrak{A}}$. By the previous consideration, the image under $h$ must be a
finitely generated $\mathfrak{A}$-submodule of $P_{\infty}$. Thus, by Lemma
\ref{lemma_FinGenSubmodulesLieInFigGenProjectiveSubmodule} we find some
$P_{\infty,0}\in P_{f}(\mathfrak{A})$ and $P_{\infty,\infty}\in P_{\oplus
}(\mathfrak{A})$ such that $P_{\infty}\simeq P_{\infty,0}\oplus P_{\infty
,\infty}$ and $\operatorname*{im}(h)\subseteq P_{\infty,0}$. Thus, we obtain a
new diagram%
\[%
\xymatrix{
& P_{\infty,0} \oplus M \oplus I_{\infty} \ar@{^{(}->}[d] \\
Y \ar@{..>}[ur] \ar[dr]_{0} \ar[r] & P_{\infty} \oplus M \oplus I_{\infty}
\ar@{->>}[d] \\
& P_{\infty,\infty}.
}%
\]
and by the universal property of kernels, we learn that $Y\rightarrow X$
factors over $Y^{\prime}:=P_{\infty,0}\oplus M\oplus I_{\infty}$, which lies
in $\mathsf{PLCA}_{\mathfrak{A},cg}$ since all summands do. This gives the
required factorization to see that $\mathsf{PLCA}_{\mathfrak{A},cg}%
\hookrightarrow\mathsf{PLCA}_{\mathfrak{A}}$ is left filtering.\newline%
\textit{(Left special)} (Step 1) Suppose $X\twoheadrightarrow X^{\prime\prime
}$ is an admissible epic with $X\in\mathsf{PLCA}_{\mathfrak{A}}$ and
$X^{\prime\prime}\in\mathsf{PLCA}_{\mathfrak{A},cg}$. Being an epic, there
exists an exact sequence%
\begin{equation}
X^{\prime}\hookrightarrow X\twoheadrightarrow X^{\prime\prime} \label{laa6}%
\end{equation}
in $\mathsf{PLCA}_{\mathfrak{A}}$. Pick PI-presentations for $X^{\prime}$ and
$X^{\prime\prime}$, where we denote the objects accordingly with a single
prime or double prime superscript. For $P^{\prime\prime}$ we may assume
$P^{\prime\prime}\in P_{f}(\mathfrak{A})$ since $X^{\prime\prime}%
\in\mathsf{PLCA}_{\mathfrak{A},cg}$. By Lemma
\ref{lemma_ExactSeqTo3By3Diagram} we may extend Equation \ref{laa6} to the
diagram%
\[%
\xymatrix{
P^{\prime} \ar@{^{(}->}[d] \ar@{^{(}->}[r] & P^{\prime} \oplus P^{\prime
\prime} \ar@{^{(}->}[d] \ar@{->>}[r] &
P^{\prime\prime} \ar@{^{(}->}[d] \\
X^{\prime} \ar@{->>}[d] \ar@{^{(}->}[r] & X \ar@{->>}[d] \ar@{->>}[r] &
X^{\prime\prime} \ar@{->>}[d] \\
I^{\prime} \ar@{^{(}->}[r] & I^{\prime} \oplus I^{\prime\prime} \ar@{->>}[r] &
I^{\prime\prime}.
}%
\]
Next, apply Lemma \ref{lemma_SplitOffInfiniteProjectivePart} to $X$ with
$F:=P^{\prime\prime}$. Write $X_{new}\in\mathsf{PLCA}_{\mathfrak{A},cg}$ for
its output $M$. We can now change the above diagram to%
\[%
\xymatrix{
P^{\prime} \ar@{^{(}->}[d] \ar@{^{(}->}[r] & P_{\infty} \oplus P_{0} \ar@
{^{(}->}[d]^{1 \oplus i} \ar@{->>}[r]^-{q} &
P^{\prime\prime} \ar@{^{(}->}[d] \\
X^{\prime} \ar@{->>}[d] \ar@{^{(}->}[r] & P_{\infty} \oplus X_{new} \ar@
{->>}[d] \ar@{->>}[r] &
X^{\prime\prime} \ar@{->>}[d] \\
I^{\prime} \ar@{^{(}->}[r] & 0 \oplus(I^{\prime} \oplus I^{\prime\prime}%
) \ar@{->>}[r] &
I^{\prime\prime}.
}%
\]
As $P^{\prime\prime}\subseteq P_{0}$, we have $q(P_{\infty})=0$ in
$P^{\prime\prime}$. Since $q$ is an admissible epic to the projective object
$P^{\prime\prime}$, the map $q$ splits, so we may decompose $P_{0}\simeq
\tilde{P}\oplus P^{\prime\prime}$ for some $\tilde{P}\in P_{f}(\mathfrak{A})$
and our diagram becomes%
\begin{equation}%
\xymatrix{
P^{\prime} \ar@{^{(}->}[d] \ar@{^{(}->}[r] & P_{\infty} \oplus({\tilde
{P}\oplus P^{\prime\prime}}) \ar@{^{(}->}[d]^{1 \oplus i} \ar@{->>}[r]^-{q} &
P^{\prime\prime} \ar@{^{(}->}[d] \\
X^{\prime} \ar@{->>}[d] \ar@{^{(}->}[r] & P_{\infty} \oplus X_{new} \ar@
{->>}[d] \ar@{->>}[r] &
X^{\prime\prime} \ar@{->>}[d] \\
I^{\prime} \ar@{^{(}->}[r] & 0 \oplus(I^{\prime} \oplus I^{\prime\prime}%
) \ar@{->>}[r] &
I^{\prime\prime}.
}
\label{laa8}%
\end{equation}
(Step 2) Following the arrows of the diagram, we see that both $P^{\prime}$ as
well as $X_{new}$ are closed submodules of $X$ ($=\left.  P_{\infty}\oplus
X_{new}\right.  $). Define%
\begin{equation}
J:=P^{\prime}\cap X_{new}\text{.} \label{laa8a}%
\end{equation}
We claim that this is a finitely generated discrete $\mathfrak{A}$-submodule
of $P^{\prime}$. The argument is the same as in the proof of Lemma
\ref{lemma_SplitOffInfiniteProjectivePart} (namely: write $C\oplus
V\hookrightarrow X_{new}\twoheadrightarrow D$ with $C$ compact, $V$ a vector
module, $D$ discrete. Then $C\cap P^{\prime}=0$ since $C$ is compact,
$P^{\prime}$ discrete, but $P^{\prime}$ is also torsionfree. So it suffices to
consider $V\cap P^{\prime}$, and since this is a closed subgroup, $J$ can only
be a lattice in $V$). Next, observe that the top row in Diagram \ref{laa8} is
actually split, i.e.%
\[
P^{\prime}\cong P_{\infty}\oplus\tilde{P}\text{,}%
\]
i.e. we can interpret $\tilde{P}$ as a submodule of $P^{\prime}$. Now apply
Lemma \ref{lemma_SplitOffInfiniteProjectivePart} to $X^{\prime}$ with
$F:=J+\tilde{P}$. Write $X_{new}^{\prime}\in\mathsf{PLCA}_{\mathfrak{A},cg}$
for its output $M$. Hence, we can rewrite the left downward column%
\[
P^{\prime}\hookrightarrow X^{\prime}\twoheadrightarrow I^{\prime}%
\]
as%
\[
P_{\infty}^{\prime}\oplus P_{0}^{\prime}\overset{1\oplus i^{\prime}%
}{\hookrightarrow}P_{\infty}^{\prime}\oplus X_{new}^{\prime}\twoheadrightarrow
0\oplus I^{\prime}\text{,}%
\]
where $J\subseteq P_{0}^{\prime}$ and $P_{0}^{\prime}\in P_{f}(\mathfrak{A})$.
By inspection of the proof of the lemma, we pick $P_{\infty}^{\prime}\oplus
P_{0}^{\prime}$ as direct summands and we can without loss of generality
assume $\tilde{P}$ to be a sub-summand appearing in $P_{0}^{\prime}$, say
$P_{0}^{\prime}\cong P_{00}^{\prime}\oplus\tilde{P}$. We can thus rewrite
Diagram \ref{laa8} as%
\begin{equation}%
\xymatrix{
P^{\prime}_{\infty} \oplus{P_{00}^{\prime}\oplus\tilde{P}} \ar@{^{(}%
->}[d]^{1 \oplus i^{\prime}} \ar@{^{(}->}[r]^-{b} & P_{\infty} \oplus
({\tilde{P}\oplus P^{\prime\prime}}) \ar@{^{(}->}[d]^{1 \oplus i} \ar@
{->>}[r]^-{q} &
P^{\prime\prime} \ar@{^{(}->}[d] \\
P^{\prime}_{\infty} \oplus X^{\prime}_{new} \ar@{->>}[d] \ar@{^{(}%
->}[r] & P_{\infty} \oplus X_{new} \ar@{->>}[d] \ar@{->>}[r] &
X^{\prime\prime} \ar@{->>}[d] \\
0 \oplus I^{\prime} \ar@{^{(}->}[r] & 0 \oplus(I^{\prime} \oplus
I^{\prime\prime}) \ar@{->>}[r] &
I^{\prime\prime}
}
\label{laa9}%
\end{equation}
such that $b$ is the inclusion of a direct summand and the identity on
$\tilde{P}$. It follows that $b$ makes $P_{\infty}^{\prime}$ a direct summand
of $P_{\infty}$ (so that $P_{\infty}\cong P_{\infty}^{\prime}\oplus
P_{00}^{\prime}$). It follows that we can compatibly remove the direct
summands $P_{\infty}^{\prime}$ resp. $P_{\infty}$ in Diagram \ref{laa9}. We
get%
\begin{equation}%
\xymatrix{
{P_{00}^{\prime}\oplus\tilde{P}} \ar@{^{(}->}[d]^{i^{\prime}} \ar@{^{(}%
->}[r] & {P_{00}^{\prime}\oplus\tilde{P}}
\oplus P^{\prime\prime} \ar@{^{(}->}[d]^{1 \oplus i}
\ar@{->>}[r]^-{q} &
P^{\prime\prime} \ar@{^{(}->}[d] \\
X^{\prime}_{new} \ar@{->>}[d] \ar@{^{(}->}[r] & P_{00}^{\prime} \oplus X_{new}
\ar@{->>}[d] \ar@{->>}[r] &
X^{\prime\prime} \ar@{->>}[d] \\
I^{\prime} \ar@{^{(}->}[r] & 0 \oplus(I^{\prime} \oplus I^{\prime\prime}%
) \ar@{->>}[r] &
I^{\prime\prime}.
}
\label{laa9b}%
\end{equation}
Now compare the middle row of the previous diagram with the middle row in the
previous diagrams: We have merely replaced $X^{\prime}$ (resp. $X$) by a
direct summand of itself. Thus, we get a commutative diagram%
\[%
\xymatrix{
X^{\prime}_{new} \ar@{^{(}->}[r] \ar[d] & P^{\prime}_{00} \oplus X_{new}
\ar@{->>}[r] \ar[d] & X^{\prime\prime} \ar@{=}[d] \\
X^{\prime} \ar@{^{(}->}[r] & X \ar@{->>}[r] & X^{\prime\prime},
}%
\]
where the top row comes from the middle row in Diagram \ref{laa9b} and the
downward arrows are the inclusions of the respective direct summands. All
objects in the top row lie in $\mathsf{PLCA}_{\mathfrak{A},cg}$. This shows
the left special property.
\end{proof}

\begin{lemma}
\label{lemma_Equiv}There is an exact equivalence of exact categories%
\[
P_{\oplus}(\mathfrak{A})/P_{f}(\mathfrak{A})\overset{\sim}{\longrightarrow
}\mathsf{PLCA}_{\mathfrak{A}}/\mathsf{PLCA}_{\mathfrak{A},cg}\text{,}%
\]
sending a projective module to itself, equipped with the discrete topology.
\end{lemma}

\begin{proof}
We clearly have an exact functor $P_{\oplus}(\mathfrak{A})\rightarrow
\mathsf{PLCA}_{\mathfrak{A}}$, basically using that $P_{\oplus}$ is a full
subcategory of the latter. Since every finitely generated projective
$\mathfrak{A}$-module has underlying abelian group $\mathbb{Z}^{n}$ for some
$n$, it is compactly generated, so we get the exact functor%
\[
P_{\oplus}(\mathfrak{A})/P_{f}(\mathfrak{A})\longrightarrow\mathsf{PLCA}%
_{\mathfrak{A}}/\mathsf{PLCA}_{\mathfrak{A},cg}\text{.}%
\]
This functor is essentially surjective: Given any $X\in\mathsf{PLCA}%
_{\mathfrak{A}}$, let $P\hookrightarrow X\twoheadrightarrow I$ be a
PI-presentation. Since $I\in\mathsf{PLCA}_{\mathfrak{A},cg}$ it follows that
$P\hookrightarrow X$ is an isomorphism in the quotient exact category
(\cite[Proposition 2.19, (2)]{MR3510209}), but $P\in P_{\oplus}(\mathfrak{A}%
)$. We next show that the functor is fully faithful: Morphisms $Y_{1}%
\rightarrow Y_{2}$ in $\mathsf{PLCA}_{\mathfrak{A}}/\mathsf{PLCA}%
_{\mathfrak{A},cg}$ are roofs%
\begin{equation}
Y_{1}\overset{e}{\twoheadleftarrow}Y_{1}^{\prime}\rightarrow Y_{2}\text{,}
\label{lx1}%
\end{equation}
where $e$ is an admissible epic with compactly generated kernel $K$. For
$Y_{1},Y_{2}$ in the strict image of the functor, these objects carry the
discrete topology. Using the structure theorem of $\mathsf{LCA}_{\mathfrak{A}%
}$ for $Y_{1}^{\prime}$, \cite[Lemma 6.5]{etnclca}, we get a decomposition%
\[
C\oplus V\hookrightarrow Y_{1}^{\prime}\twoheadrightarrow D
\]
with $C$ a compact $\mathfrak{A}$-module, $V$ a vector $\mathfrak{A}$-module
and $D$ a discrete $\mathfrak{A}$-module. Since the image of a compactum in a
discrete group is compact, it must be finite, hence torsion, but $Y_{1},Y_{2}$
are projective $\mathfrak{A}$-modules, so the image of $C$ in both
$Y_{1},Y_{2}$ must be zero. Similarly, $V$ is connected and hence its image in
$Y_{1},Y_{2}$ must be zero. Thus, without loss of generality, the roof in
Equation \ref{lx1} can be assumed to have $Y_{1}^{\prime}$ discrete, as any
roof is equivalent to such a roof. However, if $Y_{1}$ is discrete, the
compactly generated kernel $K$ must be finitely generated. Thus, as $Y_{1}$ is
projective, the epic $e$ in Equation \ref{lx1} is split and such that
$Y_{1}^{\prime}\cong Y_{1}\oplus K$ with $K$ (then by necessity) a finitely
generated projective $\mathfrak{A}$-module. Thus, the roofs representing
morphisms in $\mathsf{PLCA}_{\mathfrak{A}}/\mathsf{PLCA}_{\mathfrak{A},cg}$
are precisely the same roofs as for morphisms in $P_{\oplus}(\mathfrak{A}%
)/P_{f}(\mathfrak{A})$, and up to the same equivalence relation, proving full
faithfulness. Combining all these facts, the functor in our claim is an exact equivalence.
\end{proof}

The next proposition relies on the concept of localizing invariants in the
sense of \cite{MR3070515}.

\begin{proposition}
\label{prop_BasicSeqWithCG}Let $A$ be any finite-dimensional semisimple
$\mathbb{Q}$-algebra and $\mathfrak{A}\subseteq A$ an order. Let $\mathsf{A}$
be a stable $\infty$-category. Suppose $K:\operatorname*{Cat}_{\infty
}^{\operatorname*{ex}}\rightarrow\mathsf{A}$ is a localizing invariant with
values in $\mathsf{A}$.

\begin{enumerate}
\item There is a fiber sequence%
\begin{equation}
K(\mathfrak{A})\overset{g}{\longrightarrow}K(\mathsf{PLCA}_{\mathfrak{A}%
,cg})\overset{h}{\longrightarrow}K(\mathsf{PLCA}_{\mathfrak{A}}) \label{l_B1}%
\end{equation}
in $\mathsf{A}$. Here the map $g$ is induced from the exact functor sending a
finitely generated projective right $\mathfrak{A}$-module to itself, equipped
with the discrete topology. The map $h$ is induced from the inclusion
$\mathsf{PLCA}_{\mathfrak{A},cg}\hookrightarrow\mathsf{PLCA}_{\mathfrak{A}}$.

\item There is a morphism of fiber sequences\footnote{that is: when we write
the fiber sequences as their underlying bi-Cartesian square along with a null
homotopy for the fourth vertex, then we have a morphism of bi-Cartesian
squares, in particular the null homotopies are compatible} from Sequence
\ref{l_B1} to%
\[
K(\mathsf{Mod}_{\mathfrak{A},fg})\overset{g}{\longrightarrow}K(\mathsf{LCA}%
_{\mathfrak{A},cg})\overset{h}{\longrightarrow}K(\mathsf{LCA}_{\mathfrak{A}%
})\text{,}%
\]
based on the fully exact inclusions%
\[
P_{f}(\mathfrak{A})\subseteq\mathsf{Mod}_{\mathfrak{A},fg}\qquad
\text{and}\qquad\mathsf{PLCA}_{\mathfrak{A}}\subseteq\mathsf{LCA}%
_{\mathfrak{A}}%
\]
and the compactly generated modules respectively.
\end{enumerate}
\end{proposition}

\begin{proof}
The proof is a mild variation of \cite[Proposition 11.1]{etnclca}, but using
the fully exact subcategory $\mathsf{PLCA}_{\mathfrak{A}}$ instead of
$\mathsf{LCA}_{\mathfrak{A}}$. However, especially since the proofs are
compatible otherwise, the second claim is automatically true. For the first
claim, we set up the diagram
\begin{equation}%
\xymatrix{
K(P_{f}(\mathfrak{A})) \ar[r] \ar[d]_{g} & K(P_{\oplus}(\mathfrak{A}%
)) \ar[r] \ar[d] & K(P_{\oplus}(\mathfrak{A})/P_{f}(\mathfrak{A}%
)) \ar[d]^{\Phi} \\
K(\mathsf{PLCA}_{\mathfrak{A},cg}) \ar[r] & K(\mathsf{PLCA}_{\mathfrak{A}%
}) \ar[r] & K({\mathsf{PLCA}_{\mathfrak{A}}}/{\mathsf{PLCA}_{\mathfrak{A},cg}%
})
}
\label{l_B3}%
\end{equation}
as follows: By Proposition \ref{prop_LF1} and \ref{prop_LF2} we get fiber
sequences in $K$, forming the rows. The equivalence $\Phi$ stems from the
equivalence of the underlying exact categories, coming from Lemma
\ref{lemma_Equiv}. The downward arrows come from the exact functors sending
the respective $\mathfrak{A}$-modules to themselves, equipped with the
discrete topology. As $P_{\oplus}(\mathfrak{A})$ is closed under countable
direct sums, $K(P_{\oplus}(\mathfrak{A}))=0$ by the\ Eilenberg swindle.
\end{proof}

\section{Gorenstein orders}

\subsection{Definitions}

For any order $\mathfrak{A}\subset A$ define%
\begin{equation}
\mathfrak{A}^{\ast}:=\operatorname*{Hom}\nolimits_{\mathbb{Z}}(\mathfrak{A}%
,\mathbb{Z)}\text{.} \label{l_LinZDual}%
\end{equation}
The left $\mathfrak{A}$-module structure on this is given by%
\begin{equation}
(\alpha\cdot\varphi)(q):=\varphi(q\alpha) \label{l_LinZDual2}%
\end{equation}
(and correspondingly for the right module structure, for which we however have
no need).

\begin{example}
A general order is far from being reflexive, i.e. $\mathfrak{A}^{\ast\ast}$ is
usually strictly bigger than $\mathfrak{A}$ under the natural inclusion
$\mathfrak{A}\rightarrow\mathfrak{A}^{\ast\ast}$ (view both as submodules of
$A\overset{\sim}{\rightarrow}A^{\ast\ast}$). If $\mathfrak{A}$ is a maximal
order, the inclusion is the identity $\mathfrak{A}\overset{=}{\rightarrow
}\mathfrak{A}^{\ast\ast}$, and in our situation over the ring $\mathbb{Z}$
this is an equivalent characterization of maximality by Auslander--Goldman
\cite[(11.4) Theorem]{MR1972204}.
\end{example}

\begin{definition}
\label{def_Gorenstein}An order $\mathfrak{A}\subset A$ is called a
\emph{Gorenstein order} if one (then all) of the following properties hold:

\begin{enumerate}
\item $A/\mathfrak{A}$ is an injective left $\mathfrak{A}$-module,

\item $\operatorname*{left}$-$\operatorname*{injdim}_{\mathfrak{A}%
}(\mathfrak{A})=1$,

\item $\mathfrak{A}^{\ast}$ is a categorically compact projective
generator\footnote{sometimes this is also called a \textit{progenerator}. In
the situation at hand being categorically compact is equivalent to being a
finitely presented $\mathfrak{A}$-module.} for the category of left
$\mathfrak{A}$-modules,

\item or any of (1), (2), (3) as a right module.
\end{enumerate}
\end{definition}

The concept was introduced in \cite{MR0219527}. Most of the equivalence of
these conditions is proven in \cite[Proposition 6.1]{MR0219527}, \cite[Chapter
IX, \S 4, \S 5]{MR0283014}, while the characterization (1) is due to
Roggenkamp \cite[Lemma 5]{MR0314904}.

Non-commutative Gorenstein rings are rings with finite left and right
injective dimension, so Gorenstein orders are in particular Gorenstein rings.

We collect a few well-known facts, only in order to exhibit the usefulness of
the concept.

\begin{lemma}
\label{lemma_IntegralGroupRingsAreGorensteinOrders}For any finite group $G$,
$\mathbb{Z}[G]\subset\mathbb{Q}[G]$ is a Gorenstein order.
\end{lemma}

\begin{proof}
(\cite[Corollary 6]{MR0314904}) For any $g\in G\setminus\{e\}$ the action of
$g$ is a fixed-point free permutation of the $\mathbb{Z}$-module generators
$G$, so $\operatorname*{tr}(g)=0$, while for $g=e$ we have $\operatorname*{tr}%
(e)=\left\vert G\right\vert $. It follows that $\mathfrak{A}^{\ast}=\frac
{1}{\left\vert G\right\vert }\mathfrak{A}$ inside $\mathbb{Q}[G]$.
\end{proof}

\begin{remark}
If we want to work with group rings $\mathbb{Z}[G]\subset\mathbb{Q}[G]$ we are
basically forced to work at least in the generality of Gorenstein orders. The
slightly more specialized class of Bass orders is in general not sufficient,
\cite{MR1040944}. A group ring $\mathbb{Z}[G]$ has finite global dimension if
and only if $G=1$, so the even more specialized classes of regular or
hereditary (let alone maximal) orders are hopeless.
\end{remark}

\begin{lemma}
Any hereditary order is Gorenstein.
\end{lemma}

\begin{proof}
Consider $\mathfrak{A}\hookrightarrow A\twoheadrightarrow A/\mathfrak{A}$. As
$A$ is semisimple, $A$ is an injective $\mathfrak{A}$-module, but since
$\mathfrak{A}$ is hereditary, quotients of injectives are injective, so
$A/\mathfrak{A}$ is injective. An order is left hereditary if and only if it
is right hereditary, so there is no question about left or right here.
\end{proof}

\begin{lemma}
[{\cite[Prop. 3.6]{MR3279374}}]If $A$ is a number field, then any order of the
shape $\mathbb{Z}[\alpha]$ with $\alpha\in A$ is Gorenstein.
\end{lemma}

The paper \cite{MR3279374} also provides some examples of non-Gorenstein orders.

\subsection{Computations}

Recall that $A_{\mathbb{R}}:=\mathbb{R}\otimes_{\mathbb{Q}}A$ denotes the base
change to the reals.

\begin{proposition}
\label{prop_StdSeq}Suppose $A$ is a finite-dimensional semisimple $\mathbb{Q}%
$-algebra. If $\mathfrak{A}\subset A$ is a Gorenstein order, then%
\begin{equation}
\mathfrak{A}\hookrightarrow A_{\mathbb{R}}\twoheadrightarrow A_{\mathbb{R}%
}/\mathfrak{A} \label{lmfs3}%
\end{equation}
is a PI-presentation for $A_{\mathbb{R}}$. In particular, $A_{\mathbb{R}}%
\in\mathsf{PLCA}_{\mathfrak{A}}$.
\end{proposition}

\begin{proof}
It is clear that $\mathfrak{A}$ is a projective right $\mathfrak{A}$-module,
so we only need to show that $(A_{\mathbb{R}}/\mathfrak{A})^{\vee}$ is a
projective left $\mathfrak{A}$-module. (Step 1) First of all, we recall that
there is a non-degenerate symmetric trace pairing%
\[
\operatorname*{tr}:A\times A\longrightarrow\mathbb{Q}%
\]
on any finite-dimensional separable $\mathbb{Q}$-algebra, \cite[(9.26)
Theorem]{MR1972204}. Now define%
\begin{equation}
\widetilde{\mathfrak{A}}:=\{p\in A_{\mathbb{R}}\mid\operatorname*{tr}%
(pq)\in\mathbb{Z}\text{ for all }q\in\mathfrak{A}\}\text{.} \label{lmmf1}%
\end{equation}
This is a subset of $A_{\mathbb{R}}$ (it corresponds to the \textit{inverse
different}, \cite[p. 150]{MR1972204}). We give it the natural left
$\mathfrak{A}$-module structure induced from $A_{\mathbb{R}}$. We claim that
there is an isomorphism of left $\mathfrak{A}$-modules%
\begin{align*}
h:\widetilde{\mathfrak{A}}  &  \longrightarrow(A_{\mathbb{R}}/\mathfrak{A}%
)^{\vee}\\
p  &  \longmapsto\left(  q\mapsto e^{2\pi i\operatorname*{tr}(pq)}\right)
\text{,}%
\end{align*}
where the term on the right refers to the corresponding character on
$A_{\mathbb{R}}/\mathfrak{A}$. For the left scalar action we compute%
\[
h(\alpha p)=\left(  q\mapsto e^{2\pi i\operatorname*{tr}(\alpha pq)}\right)
=\left(  q\mapsto e^{2\pi i\operatorname*{tr}(pq\alpha)}\right)
\]
by using that $\operatorname*{tr}(xy)=\operatorname*{tr}(yx)$ for all $x,y$
(the symmetry of the trace pairing). However, the left scalar action on
characters amounts to pre-composing with the right scalar action in the
argument, see Equation \ref{lmf1}, so the character on the right agrees with
$\alpha\cdot h(p)$ as required. Next, $h$ is an isomorphism because really
$\widetilde{\mathfrak{A}}$ is just the orthogonal complement under the
Pontryagin duality pairing,%
\[
\widetilde{\mathfrak{A}}=\{p\in A_{\mathbb{R}}\mid e^{2\pi i\operatorname*{tr}%
(pq)}=1\text{ for all }q\in\mathfrak{A}\}=\mathfrak{A}^{\perp}\text{,}%
\]
so that $h$ being an isomorphism of groups is just the standard fact
$\mathfrak{A}^{\perp}\cong(A_{\mathbb{R}}/\mathfrak{A})^{\vee}$ \cite[(4.39)
Theorem]{MR3444405}. (Step 2) Next, we claim that there is an isomorphism of
left $\mathfrak{A}$-modules%
\begin{align*}
g:\widetilde{\mathfrak{A}}  &  \longrightarrow\mathfrak{A}^{\ast}\\
p  &  \longmapsto\left(  q\mapsto\operatorname*{tr}(pq)\right)
\end{align*}
(with $\mathfrak{A}^{\ast}$ as in Equation \ref{l_LinZDual}). Firstly, for the
left scalar action we find%
\[
g(\alpha p)=\left(  q\mapsto\operatorname*{tr}(\alpha pq)\right)  =\left(
q\mapsto\operatorname*{tr}(pq\alpha)\right)
\]
using the same argument as before and this is in line with the natural left
action as we had recalled in Equation \ref{l_LinZDual2}. The map $g$ is
injective. If not, we find a $p\neq0$ such that $q\mapsto\operatorname*{tr}%
(pq)$ is the zero pairing, contradicting the non-degeneracy of the trace
pairing. Surjective: Given any functional $\varphi\in\operatorname*{Hom}%
\nolimits_{\mathbb{Z}}(\mathfrak{A},\mathbb{Z)}$, by the non-degeneracy of the
trace pairing, we find some $p\in A_{\mathbb{Q}}$ such that $\varphi
(q)=\operatorname*{tr}(pq)$. Since we know that for all $q\in\mathfrak{A}$ we
have $\varphi(q)\in\mathbb{Z}$, we literally get that $p$ meets the condition
to lie in $\widetilde{\mathfrak{A}}$. (Step 3) Combining $h$ and $g$, we
obtain an isomorphism of left $\mathfrak{A}$-modules,%
\[
(A_{\mathbb{R}}/\mathfrak{A})^{\vee}\cong\mathfrak{A}^{\ast}\text{,}%
\]
but by Definition \ref{def_Gorenstein} one of the characterizations of
Gorenstein orders implies that $\mathfrak{A}^{\ast}$ is a projective left
module. This is what we had to show.
\end{proof}

\begin{definition}
Let $\mathsf{PLCA}_{\mathfrak{A},\mathbb{R}}$ be the full subcategory of
$\mathsf{PLCA}_{\mathfrak{A}}$ of objects which are also vector $\mathfrak{A}%
$-modules. In other words, this is the full subcategory whose objects have the
underlying LCA group $\mathbb{R}^{n}$ for some $n$.
\end{definition}

\begin{lemma}
\label{lemma_IdentifyPLCAVectorIC}If $\mathfrak{A}\subset A$ is a Gorenstein
order, there is an exact equivalence of exact categories%
\[
P_{f}(A_{\mathbb{R}})\overset{\sim}{\longrightarrow}\mathsf{PLCA}%
_{\mathfrak{A},\mathbb{R}}^{ic}\text{,}%
\]
sending a right $A_{\mathbb{R}}$-module to itself, equipped with the real
vector space topology. Moreover, the fully exact subcategory inclusion
$\mathsf{PLCA}_{\mathfrak{A}}\hookrightarrow\mathsf{LCA}_{\mathfrak{A}}$
induces the equality
\[
\mathsf{PLCA}_{\mathfrak{A},\mathbb{R}}^{ic}\overset{\sim}{\longrightarrow
}\mathsf{LCA}_{\mathfrak{A},\mathbb{R}}%
\]
with the category of all vector $\mathfrak{A}$-modules in $\mathsf{LCA}%
_{\mathfrak{A}}$.
\end{lemma}

\begin{proof}
Let $F(A_{\mathbb{R}})$ be the category of finitely generated free right
$A_{\mathbb{R}}$-modules. We have an exact functor%
\[
F(A_{\mathbb{R}})\longrightarrow\mathsf{PLCA}_{\mathfrak{A},\mathbb{R}}%
\]
sending $A_{\mathbb{R}}$ to itself, equipped with the real topology. We have
$A_{\mathbb{R}}\in\mathsf{PLCA}_{\mathfrak{A}}$ thanks to Proposition
\ref{prop_StdSeq}. By the $2$-functoriality of idempotent completion
\cite[\S 6]{MR2606234}, we get a unique induced exact functor $C:P_{f}%
(A_{\mathbb{R}})\longrightarrow\mathsf{PLCA}_{\mathfrak{A},\mathbb{R}}^{ic}$.
By the same argument, the inclusion%
\[
\mathsf{PLCA}_{\mathfrak{A}}\hookrightarrow\mathsf{LCA}_{\mathfrak{A}}%
\]
functorially induces an exact functor $C^{\prime}:\mathsf{PLCA}_{\mathfrak{A}%
,\mathbb{R}}^{ic}\longrightarrow\mathsf{LCA}_{\mathfrak{A},\mathbb{R}}$ since
$\mathsf{LCA}_{\mathfrak{A}}$ is already idempotent complete (as it is
quasi-abelian), and moreover the image consists only of vector modules. We
show that $C$ is essentially surjective: Every vector module $X$ is a right
$A_{\mathbb{R}}$-module, necessarily finitely generated since it must be
finite-dimensional as a real vector space. Since $A_{\mathbb{R}}$ is
semisimple, all its modules are projective and therefore $X$ is a finitely
generated projective right $A_{\mathbb{R}}$-module. Hence, $X$ is a direct
summand of some $A_{\mathbb{R}}^{n}$. However, by Proposition
\ref{prop_StdSeq} we have $A_{\mathbb{R}}\in\mathsf{PLCA}_{\mathfrak{A}%
,\mathbb{R}}$, so the idempotent completion settles the claim. Note that this
argument did not use $X\in\mathsf{PLCA}_{\mathfrak{A}}$, so it also settles
essential surjectivity of $C^{\prime}$. For $C^{\prime}$ it is clear that the
functor is fully faithful. For $C$ it follows from continuity. (More
precisely: Any $A_{\mathbb{R}}$-module homomorphism is also an $\mathbb{R}%
$-linear map and all linear maps between real vector spaces are continuous in
the real topology. Conversely, any abelian group homomorphism between uniquely
divisible groups must be a $\mathbb{Q}$-vector space map. By continuity, it
then must be an $\mathbb{R}$-linear map using the density of $\mathbb{Q}%
\subset\mathbb{R}$. Finally, this means that the $\mathfrak{A}$-module
homomorphisms are even $\mathfrak{A}\otimes_{\mathbb{Z}}\mathbb{R}%
=A_{\mathbb{R}}$ module homomorphisms)
\end{proof}

\begin{example}
We point out that this lemma would not hold without the idempotent completion.
Take $A:=\mathbb{Q}[\sqrt{2}]$, a number field. Then $\mathfrak{A}%
:=\mathbb{Z}[\sqrt{2}]$ is the ring of integers, and thus a maximal order. We
have $A_{\mathbb{R}}\simeq\mathbb{R}_{\sigma}\oplus\mathbb{R}_{\sigma^{\prime
}}$, where $\sigma,\sigma^{\prime}$ correspond to the two real embeddings
$\sqrt{2}\mapsto\pm\sqrt{2}$, giving the two possible $\mathfrak{A}$-module
structures on the reals. While $\mathbb{R}_{\sigma}$ is a vector module, we
have $\mathbb{R}_{\sigma}\notin\mathsf{PLCA}_{\mathfrak{A}}$, for otherwise
there would be a PI-presentation%
\[
P\hookrightarrow\mathbb{R}_{\sigma}\twoheadrightarrow I\text{.}%
\]
Here $P\in P_{f}(\mathfrak{A})$. As $A$ has class number one, $\mathfrak{A}$
is a principal ideal domain, so all projective $\mathfrak{A}$-modules are
free. As the underlying abelian group of $\mathfrak{A}$ is $\mathbb{Z}^{2}$,
it follows that the underlying LCA group of $P$ can only be $\mathbb{Z}^{2n}$.
On the other hand, $I$ is compact (Lemma
\ref{lemma_ObjectsInIProdCompactConnected}). However, all cocompact closed
subgroups of $\mathbb{R}$ are isomorphic to $\mathbb{Z}$. Thus, no
PI-presentation can exist.
\end{example}

\begin{corollary}
If $\mathfrak{A}\subset A$ is a Gorenstein order, all vector right
$\mathfrak{A}$-modules lie in $\mathsf{PLCA}_{\mathfrak{A}}^{ic}$, and they
are both injective and projective objects in this category.
\end{corollary}

\begin{proof}
As vector modules are projective (resp. injective) in $\mathsf{LCA}%
_{\mathfrak{A}}$ by \cite[Proposition 8.1]{etnclca}, they remain so in
$\mathsf{PLCA}_{\mathfrak{A}}$ (Lemma
\ref{lemma_ProjectivesStayProjectiveInCPI}).
\end{proof}

\begin{proposition}
\label{Prop_ExistResolutionsInFiniteTypeCase}Suppose $\mathfrak{A}\subset A$
is a Gorenstein order.

\begin{enumerate}
\item Then for every finitely generated projective right $\mathfrak{A}$-module
$P$ the sequence%
\begin{equation}
P\hookrightarrow P_{\mathbb{R}}\twoheadrightarrow P_{\mathbb{R}}/P \label{lv1}%
\end{equation}
is a PI-presentation, where $P_{\mathbb{R}}:=\mathbb{R}\otimes_{\mathbb{Z}}P$
is regarded as equipped with the real vector space topology. In particular,
$P_{\mathbb{R}}/P\in I_{\Pi}(\mathfrak{A})$.

\item Moreover, this is a projective resolution of $P_{\mathbb{R}}/P$ in
$\mathsf{PLCA}_{\mathfrak{A}}$.

\item Moreover, this is an injective resolution of $P$ in $\mathsf{PLCA}%
_{\mathfrak{A}}$.
\end{enumerate}
\end{proposition}

\begin{proof}
(1) Since $P$ is projective, there exists some $n\geq0$ and idempotent $e$
with $P=e\mathfrak{A}^{n}$. After tensoring with the reals, this cuts out the
exact sequence of Equation \ref{lv1} as a direct summand of a direct sum of
sequences of Proposition \ref{prop_StdSeq}. Thus, $(P_{\mathbb{R}}/P)^{\vee}$
is a direct summand of $(A_{\mathbb{R}}/\mathfrak{A})^{\vee}$ and thus
injective, and $P_{f}(\mathfrak{A})$ is closed under direct summands in all
right $\mathfrak{A}$-modules as well. We arrive at the said PI-presentation.
(2) As $P$ and $P_{\mathbb{R}}$ are projective objects in $\mathsf{LCA}%
_{\mathfrak{A}}$ by \cite[Proposition 8.1]{etnclca}, they remain projective in
$\mathsf{PLCA}_{\mathfrak{A}}$ by Lemma
\ref{lemma_ProjectivesStayProjectiveInCPI}, and the claim follows. (3) Use
\cite[Proposition 8.1]{etnclca} analogously.
\end{proof}

\begin{remark}
Note that all discrete modules in the above proof are finitely generated, so
we do not run into the issue that $P_{\oplus}(\mathfrak{A})$ itself need not
be idempotent complete in general (Lemma
\ref{lemma_SolvableGroupInfiniteIndecompProjective}).
\end{remark}

\begin{definition}
Let $\mathsf{PLCA}_{\mathfrak{A},\mathbb{R}D}$ be the full subcategory of
$\mathsf{PLCA}_{\mathfrak{A}}$ of objects which can be written as a direct sum%
\[
X\simeq P\oplus V
\]
with $P\in P_{\oplus}(\mathfrak{A})$ and $V$ a vector right $\mathfrak{A}$-module.
\end{definition}

\begin{lemma}
$\mathsf{PLCA}_{\mathfrak{A},\mathbb{R}D}$ is an extension-closed subcategory
of $\mathsf{PLCA}_{\mathfrak{A}}$ (and even in $\mathsf{LCA}_{\mathfrak{A}}$).
\end{lemma}

\begin{proof}
Take $\mathsf{C}:=\mathsf{LCA}_{\mathfrak{A}}$, which is weakly idempotent
complete. We want to apply Lemma \ref{lemma_ExactSeqTo3By3Diagram} to
$\mathsf{C}$ with $\mathsf{P}:=P_{\oplus}(\mathfrak{A})$ and $\mathsf{I}$ the
full subcategory of vector $\mathfrak{A}$-modules. This works since vector
modules are injective in $\mathsf{LCA}_{\mathfrak{A}}$ \cite[Proposition
8.1]{etnclca}. Every object $X\in\mathsf{PLCA}_{\mathfrak{A},\mathbb{R}D}$ has
the PI-presentation%
\[
P\hookrightarrow X\twoheadrightarrow V
\]
with respect to this choice of $\mathsf{P}$ and $\mathsf{I}$. Now let%
\[
X^{\prime}\hookrightarrow X\twoheadrightarrow X^{\prime\prime}%
\]
be an exact sequence with $X^{\prime},X^{\prime\prime}\in\mathsf{PLCA}%
_{\mathfrak{A},\mathbb{R}D}$ and $X\in\mathsf{LCA}_{\mathfrak{A}}$. Use Lemma
\ref{lemma_ExactSeqTo3By3Diagram}. It provides a PI-presentation for $X$ of
the shape%
\[
P\hookrightarrow X\twoheadrightarrow V\text{,}%
\]
with $P\in\mathsf{P}$, $V\in\mathsf{I}$, but since vector modules are also
projective \cite[Proposition 8.1]{etnclca}, this splits, giving $X\simeq
P\oplus V$, proving the claim.
\end{proof}

It follows that $\mathsf{PLCA}_{\mathfrak{A},\mathbb{R}D}$ is a fully exact
subcategory of $\mathsf{LCA}_{\mathfrak{A}}$.

\begin{lemma}
\label{lemma_E1}The category $\mathsf{PLCA}_{\mathfrak{A},\mathbb{R}}$ is left
$s$-filtering in $\mathsf{PLCA}_{\mathfrak{A},\mathbb{R}D}$.
\end{lemma}

\begin{proof}
\textit{(Left filtering)}\ If $f:V^{\prime}\rightarrow P\oplus V$ is any
morphism with $V^{\prime}\in\mathsf{PLCA}_{\mathfrak{A},\mathbb{R}}$, then
since $V^{\prime}$ is connected, we get a factorization $V^{\prime}\rightarrow
V\hookrightarrow P\oplus V$ of $f$. \textit{(Left special)} If%
\[
X^{\prime}\hookrightarrow X\twoheadrightarrow V
\]
is an exact sequence with $V\in\mathsf{PLCA}_{\mathfrak{A},\mathbb{R}}$, then
since $V$ is projective, we get a splitting, providing us with the commutative
diagram%
\[%
\xymatrix{
0 \ar@{^{(}->}[r] \ar@{->}[d] & V \ar@{->>}[r]^{1} \ar@{->}[d] & V \ar@
{=}[d] \\
X^{\prime} \ar@{^{(}->}[r] & X \ar@{->>}[r] & V \\
}%
\]
settling left specialness.
\end{proof}

\begin{lemma}
\label{lemma_E2}There is an exact equivalence of exact categories%
\[
P_{\oplus}(\mathfrak{A})\overset{\sim}{\longrightarrow}\mathsf{PLCA}%
_{\mathfrak{A},\mathbb{R}D}/\mathsf{PLCA}_{\mathfrak{A},\mathbb{R}}\text{.}%
\]

\end{lemma}

\begin{proof}
Send a module $P\in P_{\oplus}(\mathfrak{A})$ to itself, equipped with the
discrete topology. This is an exact functor. It is essentially surjective,
directly by the definition of $\mathsf{PLCA}_{\mathfrak{A},\mathbb{R}D}$.
Homomorphisms $X\rightarrow X^{\prime}$ on the right between objects in the
strict image correspond to roofs%
\[
X\overset{e}{\twoheadleftarrow}V\oplus P\rightarrow X^{\prime}%
\]
with $V$ a vector module and $e$ having vector module kernel. However, since
$V$ is connected but $X,X^{\prime}$ discrete, any such roof is trivially
equivalent to one with $V=0$. But for these the vector module kernel of $e$
must be trivial, i.e. $e$ must be an isomorphism in $\mathsf{PLCA}%
_{\mathfrak{A},\mathbb{R}D}$. Thus, any roof is equivalent to $X\overset
{1}{\twoheadleftarrow}X\rightarrow X^{\prime}$, i.e. we get just ordinary
right $\mathfrak{A}$-module homomorphisms. This shows that the functor in our
claim is fully faithful.
\end{proof}

\begin{lemma}
\label{lemma_C2Strong}Suppose%
\[
X^{\prime}\hookrightarrow V\oplus P\twoheadrightarrow V^{\prime\prime}\oplus
P^{\prime\prime}%
\]
is an exact sequence in $\mathsf{PLCA}_{\mathfrak{A}}$ whose middle and right
object lie in $\mathsf{PLCA}_{\mathfrak{A},\mathbb{R}D}$. Then $X^{\prime}%
\in\mathsf{PLCA}_{\mathfrak{A},\mathbb{R}D}$.
\end{lemma}

\begin{proof}
(Step 1) Let us work in the category $\mathsf{LCA}_{\mathfrak{A}}$. First of
all, we show that it suffices to handle the case where $V^{\prime\prime}=0$
and $P^{\prime\prime}\in P_{f}(\mathfrak{A})$. Consider%
\[
X^{\prime}\hookrightarrow V\oplus P\twoheadrightarrow V^{\prime\prime}\oplus
P^{\prime\prime}\text{.}%
\]
Note that $V^{\prime\prime}$ is a projective object in $\mathsf{LCA}%
_{\mathfrak{A}}$. Hence, there is a section $g:V^{\prime\prime}\hookrightarrow
V\oplus P$ to the epic, and since $V^{\prime\prime}$ is connected, the image
of $g$ must lie in $V$. We split off this direct summand, giving%
\begin{equation}
X^{\prime}\hookrightarrow V\oplus P\twoheadrightarrow P^{\prime\prime
}\label{lmf5}%
\end{equation}
after having changed the definition of $V$. Next, $P^{\prime\prime}$ is
projective, so we get a section $h:P^{\prime\prime}\hookrightarrow V\oplus P$.
The intersection $V\cap h(P^{\prime\prime})$ must be a discrete finitely
generated $\mathfrak{A}$-module (we refer to Equation \ref{laa3} for a
completely analogous construction, where we give a detailed argument). Thus,
by Lemma \ref{lemma_FinGenSubmodulesLieInFigGenProjectiveSubmodule} and since
$P^{\prime\prime}\in P_{\oplus}(\mathfrak{A})$ we can find a direct sum
splitting $P^{\prime\prime}\cong P_{0}^{\prime\prime}\oplus P_{\infty}%
^{\prime\prime}$ such that $h(P_{\infty}^{\prime\prime})$ lies entirely in $P$
and $P_{0}^{\prime\prime}\in P_{f}(\mathfrak{A})$. Thus, Sequence \ref{lmf5}
becomes%
\[
X^{\prime}\hookrightarrow V\oplus P\twoheadrightarrow P_{0}^{\prime\prime
}\oplus P_{\infty}^{\prime\prime}%
\]
and $h\mid_{P_{\infty}^{\prime\prime}}$ is a section for $P_{\infty}%
^{\prime\prime}$, giving%
\[
X^{\prime}\hookrightarrow V\oplus P_{0}\oplus P_{\infty}^{\prime\prime
}\twoheadrightarrow P_{0}^{\prime\prime}\oplus P_{\infty}^{\prime\prime
}\text{,}%
\]
(where $P_{0}$ denotes a complement of the image of the section) and after we
split off the summand $P_{\infty}^{\prime\prime}$, we obtain $X^{\prime
}\hookrightarrow V\oplus P_{0}\twoheadrightarrow P_{0}^{\prime\prime}$ with
$P_{0}^{\prime\prime}\in P_{f}(\mathfrak{A})$. It follows that if we prove the
claim of the lemma for this special case, it implies the general
case.\newline(Step 2) Since the underlying LCA group of $V\oplus P_{0}$ has
the shape $\mathbb{R}^{n}\oplus($discrete$)$, the closed subgroup $X^{\prime}$
must also have the shape $\mathbb{R}^{\ell}\oplus($discrete$)$ by
\cite[Corollary 2 to Theorem 7 and Remark]{MR0442141}. This direct sum
splitting on the level of LCA groups lifts to a direct sum splitting in
$\mathsf{LCA}_{\mathfrak{A}}$ by \cite[Lemma 6.1, (1)]{etnclca}, so we can
write%
\begin{equation}
X^{\prime}\cong V^{\prime}\oplus D^{\prime}\label{lmfs4}%
\end{equation}
with $V^{\prime}$ a vector $\mathfrak{A}$-module and $D^{\prime}$ discrete in
the category $\mathsf{LCA}_{\mathfrak{A}}$. Next, we apply Proposition
\ref{prop_PLCAStructureThm} to $X^{\prime}$, giving a further direct sum
decomposition $X^{\prime}\simeq P_{\infty}^{\prime}\oplus I_{\infty}^{\prime
}\oplus B^{\prime}$. We note that $I_{\infty}^{\prime}$ is compact connected
by Lemma \ref{lemma_ObjectsInIProdCompactConnected}, but by Equation
\ref{lmfs4} $X^{\prime}$ has no non-trivial compact connected subgroup at all,
so we must have $I_{\infty}^{\prime}=0$. Hence, $X^{\prime}\simeq P_{\infty
}^{\prime}\oplus B^{\prime}$. Since $P_{\infty}^{\prime}\in\mathsf{PLCA}%
_{\mathfrak{A},\mathbb{R}D}$, we conclude that the lemma is proven if we can
prove $B^{\prime}\in\mathsf{PLCA}_{\mathfrak{A},\mathbb{R}D}$.\newline(Step 3)
Thus, we may prove the claim of the lemma in the special case where
$X^{\prime}$ has a PI-presentation $P^{\prime}\hookrightarrow X^{\prime
}\twoheadrightarrow I^{\prime}$ with $P^{\prime}\in P_{f}(\mathfrak{A})$ and
$I^{\prime\vee}\in P_{f}(\mathfrak{A}^{op})$. In the isomorphism $X^{\prime
}\cong V^{\prime}\oplus D^{\prime}$ of Equation \ref{lmfs4} this implies that
$D^{\prime}$ must be a finitely generated $\mathfrak{A}$-module. Then our
sequence reads (thanks to the simplification in Step 1)%
\begin{equation}
V^{\prime}\oplus D^{\prime}\hookrightarrow V\oplus P\twoheadrightarrow
P^{\prime\prime}\label{lmf5c}%
\end{equation}
with $P^{\prime\prime}\in P_{f}(\mathfrak{A})$. We get an admissible
filtration $V^{\prime}\hookrightarrow V^{\prime}\oplus D^{\prime
}\hookrightarrow V\oplus P$ and Noether's Lemma yields the exact sequence%
\[
D^{\prime}\hookrightarrow\frac{V\oplus P}{V^{\prime}}\twoheadrightarrow
\frac{V\oplus P}{V^{\prime}\oplus D^{\prime}}%
\]
in $\mathsf{LCA}_{\mathfrak{A}}$. We note that the term on the right is
$P^{\prime\prime}$ in view of Equation \ref{lmf5c}. Moreover, the image of the
connected $V^{\prime}$ inside $V\oplus P$ will again be connected, so it must
lie in $V$. Thus, we get the exact sequence%
\[
D^{\prime}\hookrightarrow\frac{V}{V^{\prime}}\oplus P\twoheadrightarrow
P^{\prime\prime}\text{.}%
\]
Since $D^{\prime}$ and $P^{\prime\prime}$ are discrete, so must be the group
in the middle. This forces $V/V^{\prime}=0$. We get $D^{\prime}\hookrightarrow
P\twoheadrightarrow P^{\prime\prime}$. As both $P^{\prime\prime}$ and
$D^{\prime}$ are finitely generated $\mathfrak{A}$-modules, so must be $P$,
i.e. $P\in P_{f}(\mathfrak{A})$. Since $P^{\prime\prime}$ is projective, the
sequence must split, i.e. $P\cong D^{\prime}\oplus P^{\prime\prime}$. Since
$P\in P_{f}(\mathfrak{A})$ and this category is idempotent complete, we deduce
that $D^{\prime}\in P_{f}(\mathfrak{A})$. Since $X^{\prime}\cong V^{\prime
}\oplus D^{\prime}$ this implies $X^{\prime}\in\mathsf{PLCA}_{\mathfrak{A}%
,\mathbb{R}D}$ as desired.
\end{proof}

\begin{remark}
The intermediate reduction to finitely generated modules in the proof was
necessary because we used idempotent completeness and this holds for
$P_{f}(\mathfrak{A})$, but not necessarily for $P_{\oplus}(\mathfrak{A})$.
\end{remark}

\begin{lemma}
\label{lemma_C1}Suppose $\mathfrak{A}\subset A$ is a Gorenstein order. Suppose
$X\in\mathsf{PLCA}_{\mathfrak{A}}$ has a PI-presentation%
\[
P\hookrightarrow X\twoheadrightarrow I
\]
with $P\in P_{f}(\mathfrak{A})$ and $I^{\vee}\in P_{f}(\mathfrak{A}^{op})$.
Then there exists a projective resolution%
\[
P_{1}^{\prime}\hookrightarrow P_{0}^{\prime}\twoheadrightarrow X
\]
with $P_{1}^{\prime},P_{0}^{\prime}\in\mathsf{PLCA}_{\mathfrak{A},\mathbb{R}%
D}$.
\end{lemma}

\begin{proof}
(Step 1) Since $I^{\vee}\in P_{f}(\mathfrak{A}^{op})$, we apply Proposition
\ref{Prop_ExistResolutionsInFiniteTypeCase} to get an injective resolution in
$\mathsf{PLCA}_{\mathfrak{A}^{op}}$. Under Pontryagin duality, this gives us a
projective resolution%
\[
P_{1}\hookrightarrow P_{0}\overset{q}{\twoheadrightarrow}I\text{,}%
\]
where $P_{0}$ is a vector right $\mathfrak{A}$-module and $P_{1}\in
P_{f}(\mathfrak{A})$. We consider the commutative diagram%
\begin{equation}%
\xymatrix{
&                & P_1 \ar@{^{(}->}[d] \\
									    &                & P_0 \ar@{->>}[d]^{q} \ar@{-->}[dl]_{f} \\
P \ar@{^{(}->}[r]^{i} & X \ar@{->>}[r] & I,
}%
\label{lmfs1}%
\end{equation}
where we obtain the lift $f$ by exploiting that $P_{0}$ is a projective
object. Now consider the morphism $i+f:P\oplus P_{0}\rightarrow X$. Since
$i,f$ are continuous, so is $i+f$. Moreover, the map is clearly surjective.
Next, since $P$ is finitely generated and $P_{0}$ a vector module, the
underlying LCA group of $P\oplus P_{0}$ is of the shape $\mathbb{Z}^{n}%
\oplus\mathbb{R}^{m}$ for suitable $n,m\geq0$, and thus $P\oplus P_{0}$ is
$\sigma$-compact. Thus, by\ Pontryagin's Open Mapping Theorem \cite[Theorem
3]{MR0442141} $i+f$ must be an open map. Hence, $i+f$ is an admissible epic in
$\mathsf{LCA}_{\mathfrak{A}}$. Let $K$ be its kernel in $\mathsf{LCA}%
_{\mathfrak{A}}$. Consider the commutative diagram%
\begin{equation}%
\xymatrix{
0 \ar@{^{(}->}[d] \ar@{^{(}->}[r] & P \ar@{^{(}->}[d] \ar@{->>}[r]^{1}
& P \ar@{^{(}->}[d]^{i} \\
K \ar@{->}[d]^{\sim} \ar@{^{(}->}[r] & P \oplus P_{0} \ar@{->>}[d] \ar@
{->>}[r]_{i+f} & X \ar@{->>}[d] \\
K \ar@{^{(}->}[r] & P_0 \ar@{->>}[r]_{\overline{i+f}} & I
}%
\label{lmfs2}%
\end{equation}
in $\mathsf{LCA}_{\mathfrak{A}}$. It can be constructed by first setting up
the top two rows, which obviously commute, and which then gives rise to the
bottom row by a na\"{\i}ve version of the snake lemma. We note that the
quotient map $\overline{i+f}$ agrees with $q$ because any $p\in P_{0}$ can be
lifted to $(0,p)$ in $P\oplus P_{0}$ and then the remaining arrows to $I$
agree with $q$ in Diagram \ref{lmfs1}. Thus, $K$ is a kernel for $q$, which
provides us with an isomorphism $K\cong P_{0}$. It follows that $K\in
\mathsf{PLCA}_{\mathfrak{A}}$. It follows that Diagram \ref{lmfs2} is actually
a diagram in the category $\mathsf{PLCA}_{\mathfrak{A}}$. Note that the middle
row now provides a projective resolution of $X$.
\end{proof}

Define the full subcategory of modules with \emph{no small subgroups},%
\[
\mathsf{PLCA}_{\mathfrak{A},nss}:=\mathsf{PLCA}_{\mathfrak{A}}\cap
\mathsf{LCA}_{\mathfrak{A},nss}\text{,}%
\]
much in the spirit of Equation \ref{l_cat_cg}. As Pontryagin duality exchanges
groups without small subgroups with compactly generated ones, we can also
define $\mathsf{PLCA}_{\mathfrak{A},nss}$ as the Pontryagin dual of the full
subcategory $\mathsf{PLCA}_{\mathfrak{A}^{op},cg}$ of $\mathsf{PLCA}%
_{\mathfrak{A}^{op}}$. In particular, it is clear that $\mathsf{PLCA}%
_{\mathfrak{A},nss}$ is a fully exact subcategory of $\mathsf{PLCA}%
_{\mathfrak{A}}$.

\begin{corollary}
\label{cor_C1_full}Suppose $\mathfrak{A}\subset A$ is a Gorenstein order.
Every object $X\in\mathsf{PLCA}_{\mathfrak{A},nss}$ has a projective
resolution%
\[
P_{1}^{\prime}\hookrightarrow P_{0}^{\prime}\twoheadrightarrow X
\]
with $P_{1}^{\prime},P_{0}^{\prime}\in\mathsf{PLCA}_{\mathfrak{A},\mathbb{R}%
D}$.
\end{corollary}

\begin{proof}
Use Proposition \ref{prop_PLCAStructureThm} to write $X$ as $X\simeq
P_{\infty}\oplus M\oplus I_{\infty}$ such that $M$ satisfies the conditions of
Lemma \ref{lemma_C1} (and therefore has a projective resolution as required).
Further, $P_{\infty}\in P_{\oplus}(\mathfrak{A})$ is projective and lies in
$\mathsf{PLCA}_{\mathfrak{A},\mathbb{R}D}$, so this also satisfies our claim.
Finally, the underlying LCA group of $I_{\infty}$ is $\prod_{i\in\mathcal{I}%
}\mathbb{T}$ for some index set, but this has no small subgroups if and only
if $\mathcal{I}$ is finite (\cite[Theorem 2.4]{MR0215016}). In that case, and
since we know that $I_{\infty}^{\vee}\in P_{f}(\mathfrak{A}^{op})$, it follows
that $I_{\infty}$ also satisfies the conditions of Lemma \ref{lemma_C1}.
\end{proof}

\begin{theorem}
\label{thm_IdentifyNoSmallSubgroupPart}Let $\mathfrak{A}\subset A$ be a
Gorenstein order. Let $\mathsf{A}$ be a stable $\infty$-category. Suppose
$K:\operatorname*{Cat}_{\infty}^{\operatorname*{ex}}\rightarrow\mathsf{A}$ is
a localizing invariant with values in $\mathsf{A}$. Then there is an
equivalence%
\[
K(A_{\mathbb{R}})\overset{\sim}{\longrightarrow}K(\mathsf{PLCA}_{\mathfrak{A}%
,nss})\text{,}%
\]
induced from the exact functor sending a right $A_{\mathbb{R}}$-module to
itself, equipped with the real vector space topology.
\end{theorem}

\begin{proof}
(Step 1) First of all, we show that the inclusion of the fully exact
subcategory $\mathsf{PLCA}_{\mathfrak{A},\mathbb{R}D}\hookrightarrow
\mathsf{PLCA}_{\mathfrak{A},nss}$ induces an equivalence $K(\mathsf{PLCA}%
_{\mathfrak{A},\mathbb{R}D})\overset{\sim}{\longrightarrow}K(\mathsf{PLCA}%
_{\mathfrak{A},nss})$, because this exact functor induces a derived
equivalence \cite[\S 12, Theorem 12.1]{MR1421815}. The assumptions of the
cited theorem are met, because the inclusion functor satisfies (the
categorical opposite of) the axiom C1 by Corollary \ref{cor_C1_full}. Further,
it satisfies the stronger condition implying C2 by Lemma \ref{lemma_C2Strong}.
(Step 2) Next, by Lemma \ref{lemma_E1} and \ref{lemma_E2} we have the
localization fiber sequence%
\[
K(\mathsf{PLCA}_{\mathfrak{A},\mathbb{R}})\longrightarrow K(\mathsf{PLCA}%
_{\mathfrak{A},\mathbb{R}D})\longrightarrow K(P_{\oplus}(\mathfrak{A}%
))\text{,}%
\]
where $K(P_{\oplus}(\mathfrak{A}))=0$ since $P_{\oplus}(\mathfrak{A})$ is
closed under countable coproducts and we may thus apply the Eilenberg swindle.
Next, since $K$ is localizing, it is invariant under going to idempotent
completion, so the exact equivalence of exact categories $P_{f}(A_{\mathbb{R}%
})\overset{\sim}{\longrightarrow}\mathsf{PLCA}_{\mathfrak{A},\mathbb{R}}^{ic}$
of Lemma \ref{lemma_IdentifyPLCAVectorIC} induces an equivalence%
\[
K(A_{\mathbb{R}})\overset{\sim}{\longrightarrow}K(\mathsf{PLCA}_{\mathfrak{A}%
,\mathbb{R}}^{ic})\overset{\sim}{\longrightarrow}K(\mathsf{PLCA}%
_{\mathfrak{A},\mathbb{R}})
\]
in $\mathsf{A}$. Combine both results and check that the equivalence is indeed
induced by the functor claimed.
\end{proof}

\begin{theorem}
\label{thm_IdentifyCGPart}Let $\mathfrak{A}\subset A$ be a Gorenstein order.
Let $\mathsf{A}$ be a stable $\infty$-category. Suppose $K:\operatorname*{Cat}%
_{\infty}^{\operatorname*{ex}}\rightarrow\mathsf{A}$ is a localizing invariant
with values in $\mathsf{A}$. Then there is an equivalence%
\[
K(A_{\mathbb{R}})\overset{\sim}{\longrightarrow}K(\mathsf{PLCA}_{\mathfrak{A}%
,cg})\text{,}%
\]
induced from the exact functor sending a right $A_{\mathbb{R}}$-module to
itself, equipped with the real vector space topology.
\end{theorem}

\begin{proof}
Pontryagin duality is an exact functor exchanging the full subcategories of
compactly generated modules with those without small subgroups. Thus,
Proposition \ref{prop_PontryaginDualityOnPLCA} restricts to an exact
equivalence of exact categories $\mathsf{PLCA}_{\mathfrak{A},cg}\overset{\sim
}{\rightarrow}\mathsf{PLCA}_{\mathfrak{A}^{op},nss}^{op}$. Along with Theorem
\ref{thm_IdentifyNoSmallSubgroupPart} applied to $\mathfrak{A}^{op}$, we get
the two equivalences
\[
K(\mathsf{PLCA}_{\mathfrak{A},cg})\overset{\sim}{\longrightarrow
}K(\mathsf{PLCA}_{\mathfrak{A}^{op},nss}^{op})\overset{\sim}{\longrightarrow
}K(P_{f}(A_{\mathbb{R}}^{op})^{op})\text{.}%
\]
Note that if $\mathfrak{A}\subset A$ is a Gorenstein order in a semisimple
algebra, so is its opposite $\mathfrak{A}^{op}\subset A^{op}$, see Definition
\ref{def_Gorenstein}, so using Theorem \ref{thm_IdentifyNoSmallSubgroupPart}
was legitimate.\ Next, for any ring $R$ the functor $P\mapsto
\operatorname*{Hom}_{R}(P,R)$ induces an exact equivalence $P_{f}%
(R^{op})\overset{\sim}{\longrightarrow}P_{f}(R)^{op}$, relating the opposite
ring with the opposite category. Applied to $R:=A_{\mathbb{R}}$ this yields
$K(P_{f}(A_{\mathbb{R}}^{op})^{op})\overset{\sim}{\longrightarrow}%
K(P_{f}(A_{\mathbb{R}}))$. This proves our claim.
\end{proof}

\section{Main theorems}

We may now collect all our results to obtain a locally compact topological
analogue of the relative $K$-group appearing in the Burns--Flach formulation
of the ETNC with non-commutative coefficients \cite{MR1884523}.

\begin{theorem}
\label{thm_Sequence1}Suppose $A$ is a finite-dimensional semisimple
$\mathbb{Q}$-algebra and let $\mathfrak{A}\subset A$ be a Gorenstein order.
Let $\mathsf{A}$ be a stable $\infty$-category. Suppose $K:\operatorname*{Cat}%
_{\infty}^{\operatorname*{ex}}\rightarrow\mathsf{A}$ is a localizing invariant
with values in $\mathsf{A}$. Then there is a fiber sequence%
\[
K(\mathfrak{A})\longrightarrow K(A_{\mathbb{R}})\longrightarrow
K(\mathsf{PLCA}_{\mathfrak{A}})
\]
in $\mathsf{A}$. If $\mathfrak{A}$ is regular, there is a morphism of fiber
sequences to the one of \cite[Theorem 11.2]{etnclca}%
\[
K(\mathsf{Mod}_{\mathfrak{A},fg})\longrightarrow K(A_{\mathbb{R}%
})\longrightarrow K(\mathsf{LCA}_{\mathfrak{A}})\text{,}%
\]
coming from the inclusion $P_{f}(\mathfrak{A})\subseteq\mathsf{Mod}%
_{\mathfrak{A},fg}$ and $\mathsf{PLCA}_{\mathfrak{A}}\subseteq\mathsf{LCA}%
_{\mathfrak{A}}$. This morphism is an equivalence of fiber sequences.
\end{theorem}

\begin{proof}
Use Proposition \ref{prop_BasicSeqWithCG} and Theorem \ref{thm_IdentifyCGPart}%
. Unravelling the maps gives all the claims about the compatibility with
\cite{etnclca}.
\end{proof}

Finally, we may apply this to usual algebraic $K$-theory.

\begin{definition}
Suppose $A$ is a finite-dimensional semisimple $\mathbb{Q}$-algebra and let
$\mathfrak{A}\subset A$ be an order. Define%
\[
\mathsf{LCA}_{\mathfrak{A}}^{\ast}:=\mathsf{PLCA}_{\mathfrak{A}}^{ic}\text{,}%
\]
i.e. as the idempotent completion of $\mathsf{PLCA}_{\mathfrak{A}}$.
\end{definition}

\begin{theorem}
\label{thm_OrdinaryKGroupsLongExactSequence}Suppose $A$ is a
finite-dimensional semisimple $\mathbb{Q}$-algebra and let $\mathfrak{A}%
\subset A$ be a Gorenstein order. There is a long exact sequence of algebraic
$K$-groups%
\[
\cdots\rightarrow K_{n}(\mathfrak{A})\rightarrow K_{n}(A_{\mathbb{R}%
})\rightarrow K_{n}(\mathsf{LCA}_{\mathfrak{A}}^{\ast})\rightarrow
K_{n-1}(\mathfrak{A})\rightarrow\cdots
\]
for positive $n$, ending in%
\[
\cdots\rightarrow K_{0}(\mathfrak{A})\rightarrow K_{0}(A_{\mathbb{R}%
})\rightarrow K_{0}(\mathsf{LCA}_{\mathfrak{A}}^{\ast})\rightarrow
\mathbb{K}_{-1}(\mathfrak{A})\rightarrow0\text{.}%
\]
Here $\mathbb{K}_{-1}$ denotes non-connective $K$-theory. Classically, these
groups are simply called the \textquotedblleft negative $K$%
-groups\textquotedblright. Moreover,%
\[
\mathbb{K}_{n}(\mathsf{LCA}_{\mathfrak{A}}^{\ast})\cong\mathbb{K}%
_{n-1}(\mathfrak{A})
\]
for all $n\leq-1$. If $\mathfrak{A}$ is additionally a regular order (e.g.
hereditary), this sequence agrees with the one of \cite[Theorem 11.2]%
{etnclca}, and moreover $\mathbb{K}_{n}(\mathsf{LCA}_{\mathfrak{A}}^{\ast})=0$
for $n\leq-1$ in this case.
\end{theorem}

\begin{proof}
Connective $K$-theory is not a localizing invariant, so we first need to work
with non-connective $K$-theory, which we shall denote by $\mathbb{K}$,
instead. It takes values in $\mathsf{A}:=\mathsf{Sp}$, the stable $\infty
$-category of spectra. From the fiber sequence of spectra provided by Theorem
\ref{thm_Sequence1}, we obtain the long exact sequence of homotopy groups
(i.e. non-connective $K$-groups)%
\[
\cdots\rightarrow\mathbb{K}_{n}(\mathfrak{A})\rightarrow\mathbb{K}%
_{n}(A_{\mathbb{R}})\rightarrow\mathbb{K}_{n}(\mathsf{PLCA}_{\mathfrak{A}%
})\rightarrow\mathbb{K}_{n-1}(\mathfrak{A})\rightarrow\cdots\text{.}%
\]
Next, for $K$ denoting connective $K$-theory, recall that $K_{n}%
(\mathsf{C}^{ic})\cong\mathbb{K}_{n}(\mathsf{C})$ for all $n\geq0$ and any
exact category \cite{MR2206639}. The underlying category of $\mathbb{K}%
_{n}(\mathfrak{A})$ is $P_{f}(\mathfrak{A})$, which is idempotent complete, so
we deduce%
\[
K_{n}(\mathfrak{A})=\mathbb{K}_{n}(\mathfrak{A})
\]
for all $n\geq0$. The ring $A_{\mathbb{R}}$ is semisimple and in particular
any module is projective, so $K_{n}(A_{\mathbb{R}})=\mathbb{K}_{n}%
(A_{\mathbb{R}})$ for $n\geq0$, but moreover since this is a regular ring,
$\mathbb{K}_{n}(A_{\mathbb{R}})=0$ for all $n<0$. Thus, our sequence can be
rewritten as%
\[
\cdots\rightarrow K_{1}(\mathsf{PLCA}_{\mathfrak{A}})\rightarrow
K_{0}(\mathfrak{A})\rightarrow K_{0}(A_{\mathbb{R}})\rightarrow K_{0}%
(\mathsf{PLCA}_{\mathfrak{A}}^{ic})\rightarrow\mathbb{K}_{-1}(\mathfrak{A}%
)\rightarrow0
\]
as well as $\mathbb{K}_{n}(\mathsf{PLCA}_{\mathfrak{A}})\cong\mathbb{K}%
_{n-1}(\mathfrak{A})$ for $n\leq-1$.
\end{proof}

The case of group rings is of particular relevance.

\begin{corollary}
Suppose $G$ is a finite group. Take $A=\mathbb{Q}[G]$ and $\mathfrak{A}%
:=\mathbb{Z}[G]$. There is a long exact sequence of algebraic $K$-groups%
\[
\cdots\rightarrow K_{n}(\mathbb{Z}[G])\rightarrow K_{n}(\mathbb{R}%
[G])\rightarrow K_{n}(\mathsf{LCA}_{\mathbb{Z}[G]}^{\ast})\rightarrow
K_{n-1}(\mathbb{Z}[G])\rightarrow\cdots
\]
for positive $n$, ending in%
\[
\cdots\rightarrow K_{0}(\mathbb{Z}[G])\rightarrow K_{0}(\mathbb{R}%
[G])\rightarrow K_{0}(\mathsf{LCA}_{\mathbb{Z}[G]}^{\ast})\rightarrow
\mathbb{K}_{-1}(\mathbb{Z}[G])\rightarrow0
\]
and $\mathbb{K}_{n}(\mathsf{LCA}_{\mathbb{Z}[G]}^{\ast})=0$ for $n<0$.
\end{corollary}

The group $\mathbb{K}_{-1}(\mathbb{Z}[G])$ is well-understood by the work of
Carter. He has shown that%
\[
\mathbb{K}_{-1}(\mathbb{Z}[G])\cong\mathbb{Z}^{a}\oplus(\mathbb{Z}/2)^{b}%
\]
for suitable $a,b\in\mathbb{Z}_{\geq0}$, which are a little involved to
describe explicitly, \cite[Theorem 1]{MR590500}. A lot of explicit
computations can be found for example in \cite{MR2600138}, \cite{MR3169512}.
Although this shows that some literature and research exists, it appears that
in general the study of negative $K$-groups of orders in semisimple algebras
is not very developed.

\begin{proof}
Apply Theorem \ref{thm_OrdinaryKGroupsLongExactSequence}. This is possible
because $\mathbb{Z}[G]$ is a Gorenstein order by Lemma
\ref{lemma_IntegralGroupRingsAreGorensteinOrders}. Moreover, $\mathbb{K}%
_{n}(\mathbb{Z}[G])=0$ for $n\leq-2$ by work of Carter \cite{MR590500},
\cite{MR561543}.\footnote{We remark that Hsiang has conjectured that
$\mathbb{K}_{n}(\mathbb{Z}[G])=0$ for $n\leq-2$ for any finitely presented
group. This remains open.}
\end{proof}

\begin{corollary}
The non-connective $K$-theory spectrum $\mathbb{K}(\mathsf{LCA}_{\mathbb{Z}%
[G]}^{\ast})$ for the integral group ring of any finite group is actually connective.
\end{corollary}

Finally, let us discuss the analogue of the comparison map in \cite{etnclca2}.
We refer to that paper for background on the terms and notation we employ.

\begin{theorem}
\label{thm_ComparisonMap}Let $A$ be a finite-dimensional semisimple
$\mathbb{Q}$-algebra and $\mathfrak{A}\subset A$ any order. Then the map%
\begin{equation}
\vartheta:K_{0}(\mathfrak{A},\mathbb{R})\longrightarrow K_{1}(\mathsf{LCA}%
_{\mathfrak{A}}^{\ast}) \label{lct5}%
\end{equation}
sending the Bass--Swan representative $[P,\varphi,Q]$ to the double exact
sequence $\left\langle \left\langle P,\varphi,Q\right\rangle \right\rangle $
(as defined in \cite{etnclca2}) is a well-defined morphism from the Bass--Swan
to the Nenashev presentation. If $\mathfrak{A}$ is a Gorenstein order, then
this map is an isomorphism.
\end{theorem}

\begin{proof}
One can adapt the proof of \cite{etnclca2} with only a few changes. First of
all, note that all objects which occur in the Nenashev representative
$\left\langle \left\langle P,\varphi,Q\right\rangle \right\rangle $ lie in the
full subcategory $\mathsf{PLCA}_{\mathfrak{A}}\subset\mathsf{LCA}%
_{\mathfrak{A}}$, so the map naturally lands in $K_{1}(\mathsf{LCA}%
_{\mathfrak{A}}^{\ast})=K_{1}(\mathsf{PLCA}_{\mathfrak{A}})$. Moreover, all
the proofs that the map is well-defined carry over verbatim. This already
suffices to show that the map exists. It only remains to prove that it is an
isomorphism if $\mathfrak{A}$ is Gorenstein. To this end, we also copy the
proof of \cite{etnclca2}. Replace the diagram in the statement of
\cite[Theorem 3.2]{etnclca2} by%
\[%
\xymatrix{
\cdots\ar[r] & K_{1}(\mathfrak{A},{\mathbb{R} }) \ar[r] \ar[d] & K_{1}%
(\mathfrak{A}) \ar@{=}[d] \ar[r] & K_{1}(A_{{\mathbb{R} }}) \ar@{=}%
[d] \ar[r]^{\delta} & K_{0}(\mathfrak{A},{\mathbb{R} }) \ar[d]^{\vartheta}
\ar[r]  & K_{0}(\mathfrak{A}) \ar@{=}[d] \ar[r] & \cdots\\
\cdots\ar[r] & K_{2}(\mathsf{LCA}_{\mathfrak{A}}) \ar[r] & K_{1}(\mathfrak
{A}) \ar[r] & K_{1}(A_{{\mathbb{R} }}) \ar[r] & K_{1}(\mathsf{LCA}^{\ast
}_{\mathfrak{A}}) \ar[r]_{\partial} & K_{0}(\mathfrak{A}) \ar[r] & \cdots
\text{,} \\
}%
\]
where $\vartheta$ is the map of Equation \ref{lct5} and the bottom row is the
one coming from Theorem \ref{thm_OrdinaryKGroupsLongExactSequence}. Then
proceed in the proof exactly as loc. cit., except for the following changes:
The diagram%
\[%
\xymatrix{
{\mathsf{Mod}_{{\mathfrak{A}},fg}} \ar[r] \ar[d]_{g} & {\mathsf{Mod}%
_{\mathfrak{A}}} \ar[r] \ar[d] & {{\mathsf{Mod}_{\mathfrak{A}}}}%
/{{\mathsf{Mod}_{{\mathfrak{A}},fg}}} \ar[d]^{\Phi} \\
\mathsf{LCA}_{\mathfrak{A},cg} \ar[r] & \mathsf{LCA}_{\mathfrak{A}}
\ar[r] & {\mathsf{LCA}_{\mathfrak{A}}}/{\mathsf{LCA}_{\mathfrak{A},cg}}
}%
\]
needs to be replaced by the one of categories underlying Diagram \ref{l_B3}.
The exact equivalence of exact categories%
\[
\mathsf{Mod}_{\mathfrak{A}}/\mathsf{Mod}_{\mathfrak{A},fg}\overset{\sim
}{\longrightarrow}\mathsf{LCA}_{\mathfrak{A}}/\mathsf{LCA}_{\mathfrak{A},cg}%
\]
needs to be replaced by $P_{\oplus}(\mathfrak{A})/P_{f}(\mathfrak{A}%
)\overset{\sim}{\longrightarrow}\mathsf{PLCA}_{\mathfrak{A}}/\mathsf{PLCA}%
_{\mathfrak{A},cg}$ of Lemma \ref{lemma_Equiv}. The rest works verbatim,
always just using that all the objects which the proof uses already lie in the
full subcategory $\mathsf{LCA}_{\mathfrak{A}}^{\ast}$ of $\mathsf{LCA}%
_{\mathfrak{A}}$.
\end{proof}

\bibliographystyle{amsalpha}
\bibliography{ollinewbib}

\end{document}